\documentclass[12pt,reqno]{amsart}

\title[Stochastic derivatives and generalized $h$-transforms]{Stochastic derivatives and generalized $h$-transforms of Markov processes}
\author{Christian L\'eonard}
\date{\today}

\usepackage{amssymb, amsmath, amsfonts, latexsym, enumerate}

\newtheorem{theorem}{Theorem}
\newtheorem{lemma}[theorem]{Lemma}
\newtheorem{proposition}[theorem]{Proposition}
\newtheorem{corollary}[theorem]{Corollary}
\newtheorem{result}[theorem]{Result}

\newtheorem{definition}[theorem]{Definition}

\theoremstyle{remark}
\newtheorem{remark}[theorem]{Remark}
\newtheorem{remarks}[theorem]{Remarks}

\numberwithin{theorem}{section}

\newcommand{\RR}{\mathbb{R}}
\newcommand\Rd{\RR^d}
\newcommand{\1}{\mathbf{1}}

\newcommand\pf{_{\#}}
\newcommand\as{\textrm{-a.s.}}
\renewcommand\ae{\textrm{-a.e.}}

    \DeclareMathOperator{\dom}{dom}

    \DeclareMathOperator{\supp}{supp}

\newcommand{\boulette}[1]{$\bullet$\ Proof of #1.}
\newcommand{\Boulette}[1]{\par\medskip\noindent $\bullet$\ Proof of #1.}

\newcommand\seq[2]{(#1_#2)_{#2\ge1}}
\newcommand\Seq[2]{(#1^#2)_{#2\ge1}}
\newcommand\Lim[1]{\lim_{#1\rightarrow\infty}}

\newcommand\Limh{\lim_{h\downarrow0}}

\newcommand{\lsc}{lower semicontinuous}

\newcommand{\cadlag}{c\`adl\`ag}

\newcommand\XX{\mathcal{X}}

\newcommand\OO{\Omega}
\newcommand\ii{{[0,1]}}
\newcommand\iX{{\ii\times\XX}}
\newcommand\iO{{\ii\times\OO}}
\newcommand{\iR}{\ii\times\Rd}

\newcommand\PX{\mathrm{P}(\XX)}

\newcommand\PO{\mathrm{P}(\OO)}
\newcommand\IX{\int_{\XX}}

\newcommand\IO{\int_\OO}

\newcommand\IiX{\int_{\iX}}
\newcommand\IR{\int_{\Rd}}
\newcommand\UR{\mathcal{U}_R}
\newcommand\GR{\Gamma^R}
\newcommand\LR{\mathcal{L}^R}
\newcommand\GP{\Gamma^P}
\newcommand\LP{\mathcal{L}^P}

\newcommand\XXX[2]{X_{[#1,#2]}}

\newcommand\II[2]{\int_{[#1,#2]}}

\renewcommand\lg[2]{#1_{#2^-}}

\newcommand\Xb{\overline{X}}
\newcommand\Pb{\overline{P}}
\newcommand\dPR{\frac{dP}{dR}}

\newcommand\BX{\mathrm{B}(\XX)}

\newcommand\lo{{\lambda_o}}
\newcommand{\DD}{D_\RR}
\newcommand\Gt{\widetilde{G}} 
\newcommand\Vt{\widetilde{V}}

\newcommand{\MP}{\mathrm{MP}}

\newcommand{\CCc}{\mathcal{C}^{1,2}_c(\iR)}
\newcommand{\Cc}{\mathcal{C}^{2}_c(\Rd)}
\newcommand{\nablat}{{\widetilde\nabla}}

\newcommand{\xx}{\mathrm{x}}

\renewcommand{\gg}{g_1}
\newcommand{\ff}{f_0}
\newcommand{\gs}{{\gamma^*}}
\newcommand{\Lg}{{L^\gamma(m)}}
\newcommand{\Lgs}{{L^{\gamma^*}(m)}}
\newcommand{\Lexp}{{L^{\exp}(m)}}
\newcommand{\LlogL}{{L\log L(m)}}

\newcommand{\NgR}{\|_{L^{\gamma}(R)}} 
\newcommand{\Ngm}{\|_{L^\gamma(m)}}

\begin{document}


 \address{Modal-X. Universit\'e Paris Ouest. B\^at.\! G, 200 av. de la R\'epublique. 92001 Nanterre, France}
 \email{christian.leonard@u-paris10.fr}

\begin{abstract}
Let $R$ be a continuous-time Markov process on the time interval $\ii$  with values in some state space $\XX.$  We transform this reference process $R$ into $P:=\ff(X_0)\exp \left(-\II01 V_t(X_t)\,dt\right)\gg(X_1)\,R$ where $\ff,\gg$ are nonnegative measurable functions on $\XX$ and $V$ is some measurable function on $\iX.$ It is easily seen that $P$ is also Markov. The aim of this paper is to identify the Markov generator of $P$ in terms of the Markov generator of $R$ and of the additional ingredients: $\ff,\gg$ and $V$ in absence of regularity assumptions on $\ff,\gg$ and $V.$ 
\\
As a first step, we show that the extended generator of a Markov process is essentially its stochastic derivative. Then, we compute the stochastic derivative of $P$ to identify its generator, under a finite entropy condition. The abstract results are illustrated with continuous diffusion processes on $\Rd$ and  Metropolis algorithms on a discrete space.
\end{abstract}

\maketitle
    \tableofcontents


\section{Introduction}

We consider continuous-time Markov processes with values in some Polish space $\XX$ equipped with its Borel $\sigma$-field. 

\subsection*{Notation}
Let us fix some notation. The path space is the set $$\OO=D(\ii,\XX)$$ of all right continuous and left limited  (\cadlag)  $\XX$-valued trajectories $\omega=(\omega_t)_{t\in\ii}\in\OO.$ It is equipped with the cylindrical $\sigma$-field:  $\sigma(X_t;t\in\ii)$ which is generated by the canonical process $X=(X_t)_{t\in\ii}$ defined for each $t\in\ii$ and $\omega\in\OO$ by $X_t(\omega)=\omega_t\in\XX.$
We denote $\PO$ the set of all probability measures on $\OO.$ As usual, we call process any $P\in\PO$ or any random element of $\OO$ as well. For any $\mathcal{T}\in\ii,$ we denote $X_\mathcal{T}=(X_t)_{t\in \mathcal{T}}$ and the push-forward measure $P_\mathcal{T}=(X_\mathcal{T})\pf P.$ In particular, for any
 $0\le r\le s\le1,$ $X_{[r,s]}=(X_t)_{r\le t\le s}$, $P_{[r,s]}=(X_{[r,s]})\pf P$ and $P_t=(X_t)\pf P\in\PX$ denotes the law of the position $X_t$ at time $t$ where $\PX$ the set of all probability measures on $\XX.$ The filtration is the canonical one: $\left(\sigma(X_{[0,t]}); t\in\ii\right).$

\subsection*{Aim of the article}

Let $R$ be the law of some nicely behaved Markov process. We take this probability measure $R$ as our \emph{reference} law (this explains its unusual name $R$) and call \emph{generalized $h$-transform} of $R$,  any  $P\in\PO$ which is absolutely continuous with respect to $R:$ $P\ll R,$ and with its Radon-Nikodym derivative of the special form:
\begin{equation}\label{eq-26}
P=\ff(X_0)\exp\left(-\II01 V_t(X_t)\,dt\right)\gg(X_1)\ R
\end{equation}
where $\ff,\gg:\XX\to[0,\infty)$ are nonegative Borel measurable functions on $\XX$,  the potential $V:\iX\to\RR$ is also assumed to be Borel measurable on $\iX$
 and all of them satisfy integrability conditions  such that \eqref{eq-26} defines a probability measure. We also say for short that $P$ is an \emph{$h$-process}.

It is easy to show (Proposition \ref{res-10} below) that $P$ inherits the Markov property from $R.$ Consequently, it is tempting to know more about its infinitesimal generator. The aim of this article is to derive the generator of the Markov process $P$ without assuming too many regularity conditions on $R$, $\ff, \gg$ and $V.$

\subsection*{Usual $h$-transform}

Motivated by potential theory, the special case when $V\equiv 0$ but the terminal time $t=1$ is replaced by some stopping time $\tau:$
\begin{equation*}
P=\ff(X_0)h(X_\tau)\ R^{(\tau)},
\end{equation*}
 has been introduced in 1957 by J.L. Doob  \cite{Doob57,Doob84} with $R^{(\tau)}$ a  Wiener process $R$ killed at  the exit time $\tau$ of  a bounded domain $D$ of $\Rd$.  In this situation, for all $t\ge0$ and $x$ in  $D,$ the transition probability distributions of $P$ are given by
$$
	P(X_t\in dz\mid X_0=x)\propto h_t(z)\,R_t^{(\tau)}(dz\mid X_0=x)
$$
where $\propto$ means ``proportional to'' and  $z\mapsto h_t(z)=E_R[\1_{\{\tau>t\}} h(X_\tau)\mid X_t=z]$ is a \emph{space-time harmonic} function on $D;$ this explains the letter $h$. 

\subsection*{An example}

In this paper, we shall only be concerned with the transform defined by \eqref{eq-26}, without stopping times. As an example, suppose that the reference process $R$ is the unique solution of some stochastic differential equation
\begin{equation*}
dX_t=b(X_t)\,dt+\sigma(X_t)\,dW_t
\end{equation*}
with locally Lipschitz coefficients $b$ and $\sigma,$ where $W$ is a standard Wiener process on $\XX=\Rd.$ This implies that $R$ is a solution of the martingale problem $\MP(b,a):$
$$
R\in\MP(b,a),
$$ 
with $b$ an adapted (drift) vector field and $a=\sigma \sigma^*$ an adapted (diffusion) matrix field. Since $P\ll R,$ Girsanov's theory tells us  that there exists some adapted vector field $\beta$ such that $P$ solves 
$$
P\in\MP(b+a \beta,a).
$$
Now the problem is to express $\beta$ in terms of the ingredients $a,b,\ff,\gg$ and $V.$ Specifying the abstract results of this article to  this continuous diffusion case leads to the next result (see Theorem \ref{res-18} below): The additional drift term $\beta$ can be written as
\begin{equation}\label{eq-28}
\beta(t,x)=\nablat^P	\psi(t,x),\quad dtP_t(dx)\ae
\end{equation}
where 
\begin{equation}\label{eq-27}
\psi(t,x):=\log E_R\Big[\exp \Big(-\II t1 V_s(X_s)\,ds\Big)\gg(X_1)\mid X_t=x\Big],\quad dtP_t(dx)\ae
\end{equation}
is defined $dtP_t(dx)\ae$ and $\nablat^P$ is some linear operator which we call the \emph{$P$-extended gradient}. This gradient coincides with the usual one on smooth functions: $\nablat^P u=\nabla u,$ for all $u\in\Cc,$ when the diffusion matrix $a$ has full rank. Of course, if $R$ admits a regularizing and positivity improving transition probability density (for instance if $R$ is the Wiener measure) and $V=0,$ then $\psi(t,x):=\log E_R(\gg(X_1)\mid X_t=x)$ is well-defined and smooth on $[0,1)\times \Rd$ and $\beta=\nabla \psi$. This situation is investigated in details by H. F\"ollmer \cite{Foe85}.  On the other hand, when $V$ is a non-regular measurable function, even if $R$ admits a regularizing semigroup, $\psi$ may be a non-regular continuous function and \eqref{eq-28} has an unusual meaning.

\subsection*{Non-regularity of $V$}

The transition probability distributions in both directions of time of the generalized $h$-transform $P$  are the Euclidean analogues \cite{CZ91,CZ08}  of  the Feynman propagators \cite{FH65}  in the sense that for all $t\in\ii$ 
\begin{eqnarray*}
P(X_t\in dz\mid X_0=x)&\propto&E_R \left[\exp \left(-\II t1 V_s(X_s)\,ds\right)\gg(X_1)\mid X_t=z\right]\,R(X_t\in dz\mid X_0=x)\\ 
P(X_t\in dz\mid X_1=y)&\propto&
E_R \left[\ff(X_0)\exp \left(-\II 0t V_s(X_s)\,ds\right) \mid X_t=z\right]\,R(X_t\in dz\mid X_1=y).
\end{eqnarray*}
As non-regular potentials $V$ are usual in physics, for instance discontinuous potentials with vertical asymptotic directions, we do not even assume that $V$ is  continuous. 

From another view point, \eqref{eq-26} is the generic form of the solution of the minimizer of the relative entropy 
$$
H(P|R):=\IO \log \left(\dPR\right)\,dP\in [0,\infty]
$$
which is seen as a function of $P$, subject to the constraints that its initial law $P_0$ is equal to some given $\mu_0\in\PX$ and its flow of time-marginal laws $(P_t)_{t\in\ii}$ solves some prescribed Fokker-Planck evolution equation. In this convex optimization  problem,  $\ff,\gg$ and $V$ act like Lagrange multipliers. See \cite{Csi75,Foe85,CL94,CL95,CL96} for  related entropy minimization problems and \cite{Leo01b} for a convex analytic derivation of this statement.  For instance, when  motivated by stochastic mechanics \cite{Nel88}, the above mentioned Fokker-Planck equation is related to the solution  of some Schr\"odinger equation and its drift term  explodes on the (nodal) set where the wave function  vanishes. This enforces  irregularities of $V.$ See the introduction of \cite{MZ85} for a brief explanation of this point and also Eq. (8) of  \cite{MZ85} where the potential $V_t(x)=\frac{A^R \Phi_t}{\Phi_t}(x)$ appears, with $A^R$  the Markov generator of $R$ and $\Phi$ the wave function.

\subsection*{Previous approaches to this problem}

Let $P\in\PO$ be a Markov process and $T_{s,t}^Pu(x):=E_P[u(X_t)\mid X_s=x],$ $u\in U,$ $0\le s \le t,$ be its semigroup on some Banach function space $(U,\|\cdot\|_U).$ For instance $U$ may be the space of all  bounded Borel measurable functions on $\XX$ equipped with the topology of uniform convergence. Its infinitesimal generator is $A^P=(A^P_t)_{t\in\ii}$ with
\begin{equation}\label{eq-41}
	A^P_tu(x):=\|\cdot\|_U\hbox{-}\Limh \frac 1h E_P[u(X_{t+h})-u(X_t)\mid X_t=x],
	\quad u\in\dom A ^P
\end{equation}
where the domain $\dom A^P$ of $A^P$ is precisely the set of all functions $u\in U$ such that the above strong limit exists for all $t\in [0,1)$ and $x\in\XX.$
We have seen with \eqref{eq-28} and \eqref{eq-27} that the function $g$ defined by
\begin{equation}\label{eq-29}
	g_t(x):=E_R \left[\exp \left(-\II t1 V_s(X_s)\,ds\right)\gg(X_1)\mid X_t=x\right],\quad dtP_t(dx)\ae
\end{equation}
plays an important role in the description of the dynamics of $P.$  
One can prove rather easily (see \cite{RY99} for instance) that when $g$ is positive and regular enough,  the generator $A^P$ of the Markov semigroup associated with $P$ is given for regular enough functions $u$ on $\XX,$ by
\begin{equation}\label{eq-40}
A^P_t u(x)=A^R u(x)+\frac{\Gamma(g_t,u)}{g_t}(t,x),\quad (t,x)\in\iX
\end{equation}
where $\Gamma$ is the carr\'e du champ operator, defined for all functions $u,v$ such that $u,v$ and the product $uv$ belong to the domain $\dom A^R$ of $A^R,$ by
$$
	\Gamma(u,v)=A^R(uv)-u A^Rv-vA^Ru.
$$
For Eq.\ \eqref{eq-40} to be meaningful, it is necessary that for all $t\in\ii,$ $g_t$ and the product $g_tu$ belong to $\dom A^R.$ 
But we have already noticed that with a non-regular potential $V$, $g$ might be non-regular as well. There is no reason why $g_t$ and $g_tu$ are in $\dom A^R$ in general.

Clearly, one must drop the semigroup approach and work with semimartingales or Dirichlet forms. The Dirichlet form theory is natural for constructing irregular processes and has been employed in similar contexts, see \cite{Alb00}.  But it is made-to-measure  for reversible processes and not very efficient when going   beyond reversibility. Let us have a look at the semimartingale approach. Working with semimartingales means that instead of the infinitesimal semigroup generators $A^R$ and $A^P$, we consider \emph{extended generators} in the sense of the Strasbourg school \cite{DM4}, see Definition \ref{def-05} below. This natural idea has already been implemented by P.-A. Meyer and W.A. Zheng \cite{MZ84, MZ85} in the context of stochastic mechanics and also by  P. Cattiaux and the author in \cite{CL94,CL96} for solving  related entropy minimization problems. But one still had to face the remaining problem of giving some sense to $\Gamma(g_t,u)$. Consequently, restrictive assumptions were imposed: reversibility in \cite{MZ85}  and,  in \cite{CL96}, the standard hypothesis that the domain of the extended generator of $R$ contains a ``large'' subalgebra. In practice this last requirement is not easy to verify, except for standard regular processes. In particular, it is difficult to find criteria  for this property to be inherited by $P$ when $P\ll R.$
\\
In the present article, we overcome these limitations by choosing a  different strategy which is based on stochastic derivatives and in some sense is more direct. 

\subsection*{Further developments}

Generalized $h$-processes are not only designed for Euclidean quantum mechanics \cite{CZ08} or stochastic mechanics \cite{Nel88}.
\begin{enumerate}[(i)]
\item
They are a valuable tool for obtaining a new  look at Hamilton-Jacobi-Bellman equations, by comparing the definition \eqref{eq-26} with the usual Girsanov exponential Radon-Nikodym density.
\item
Because of the time symmetry of their definition when $R$ is assumed to be reversible, they may bring interesting information about time reversal. 
\item
Even when $V$ is zero, \eqref{eq-26} provides an interesting process $P$ which is sometimes called a Schr\"odinger bridge. It minimizes $H(P|R)$ subject to the  marginal constraints $P_0=\mu_0$ and $P_1=\mu_1$. A connection with optimal transport is described in \cite{Leo10a}. The flow $(P_t)_{t\in\ii}$ of this bridge is similar to the displacement interpolation introduced by R. McCann \cite{McC95} which is used for deriving functional inequalities or as a heuristic guideline in the so-called Otto calculus, see \cite{Vill09}. This suggests that  using $(P_t)_{t\in\ii}$ instead of the displacement interpolation could yield interesting results.
\end{enumerate}
These potential developments will be investigated in future works.

\subsection*{Outline of the paper}

The \emph{stochastic derivative}  $L^P$ of $P:$
$$
	L^P_t u(x):=\Limh \frac1h E_P \left[u(X_{t+h})-u(X_t)\mid  X_t=x\right],\quad u\in\dom L^P
$$
(compare \eqref{eq-41}) was introduced by E. Nelson in \cite{Nel67}. As usual, $\dom L^P$ is defined to be the set of all functions $u$ such that the above limit exists, for the exact definition see Definition \ref{def-04}.

As a first step, we show that for a Markov process $P$, the stochastic derivative  is equal to the extended generator $\LP$ on a large class of  functions $u$ on $\XX:$ 
$$
\LP_t u(x)=L^P_tu(x),\quad dtP_t(dx)\ae
$$
This identity is the purpose of next Section \ref{sec-SD} whose main results are Theorem \ref{res-02} and Proposition \ref{res-03}. The key of Theorem \ref{res-02}'s proof is the convolution Lemma \ref{res-21}.

With this general tool at hand, it remains to compute $L^Pu$ for sufficiently many functions $u$ to determine the martingale problem associated with $P.$ And in  view of \eqref{eq-40}, with $g_t$  defined at \eqref{eq-29}, this essentially amounts to : 
\begin{enumerate}
\item[(i)]
Prove that $g_t\in\dom L^R$ and compute $L^Rg_t$;
\item[(ii)]
Prove that $g_t u\in\dom L^R$  for many ``regular'' functions $u.$
\end{enumerate}

Problem (i) is solved at Section \ref{sec-FK} by means of standard integration technics.

Problem (ii) is trickier. We solve it at Section \ref{sec-Gh} by assuming that the relative entropy of $P$ with respect to $R$ is finite:
$$
H(P|R)<\infty.
$$
The main technical step for solving this problem is Lemma \ref{res-07} which allows us not to rely on Girsanov's theory in its usual form. In particular our abstract results are valid without assuming that $R$ has the representation property (any $R$-martingale can be represented as some stochastic integral). 

The main result of this paper is Theorem \ref{res-09}. It extends \eqref{eq-40}.

At Sections \ref{sec-diffusion} and \ref{sec-chain} we examplify  our abstract results by means of continuous diffusion processes on $\Rd$ and time-continuous Markov chains. The main results of these sections are Theorem \ref{res-18} which states \eqref{eq-28} and Theorem \ref{res-20} which describes the dynamics of the $h$-transforms of \emph{Metropolis algorithms} on a discrete countable state space $\XX.$

\section{Stochastic derivatives}\label{sec-SD}

We denote for any $t\in\ii,$ $\Xb_t:=(t,X_t)\in\iX$ and for any stopping time $Y^\tau_t:=Y_{t\wedge \tau}$ and $\Xb^\tau_t:=(t\wedge \tau,X_{t\wedge \tau}).$

Let $P$ be a probability measure on $\OO.$ Recall that a process
$M$ is called a \emph{local $P$-martingale} if there exists a
sequence $\seq\tau k$ of $\ii\cup\{\infty\}$-valued stopping times
such that $\Lim k\tau_k=\infty,$ $P\as$ and for each $k\ge1,$ the
stopped process $M^{\tau_k}$ is a uniformly integrable
$P$-martingale. A process $Y$ is called a \emph{special
$P$-semimartingale} if $Y=B+M,$ $P\as$ where $B$ is a predictable
bounded variation process and $M$ is a local $P$-martingale.

\begin{definition}[Nice semimartingale]\label{def-01}
A process $Y$ is called a nice\footnote{This is a ``local'' definition in the sense that this notion  probably appears somewhere else with another name.} $P$-semimartingale if $Y=B+M$ is a
special $P$-semimartingale and the bounded variation process $B$
has absolutely continuous sample paths $P\as$
\end{definition}

\begin{definition}[Extended generator of a Markov process]\label{def-05}
Let $P$ be a Markov process. A measurable function $u$ on $\iX$ is said to be in
the domain of the extended generator of $P$ if there exists a
measurable function $v$ on $\iX$  such that
 $\II01|v(t, X_t)|\,dt<\infty,$ $P\ae$ and the process $$M^u_t:=u(t,X_t)-u(0,X_0)-\II0t v(s,
    X_s)\,ds,\quad 0\le t\le1,$$ is a local $P$-martingale.
We denote
$$
    v(t,x)=:\LP u(t,x)
$$
and call $\LP $ the extended generator of $P.$ The domain
of the extended generator of $P$ is denoted by $\dom
\LP .$
\end{definition}

\begin{remarks}\label{rem-02}\ \begin{enumerate}[(a)]
\item
In other words, the measurable function $u$ on $\iX$ is in
$\dom\LP $ if the process $u(t,X_t)$ is a nice
$P$-semimartingale.
\item
The adapted process $t\mapsto\II0t
v(s, X_s)\,ds$ is predictable since it is continuous. 
\item
$M^u$ admits a \cadlag\ $P$-version as a local $P$-martingale (we always choose this regular version).
\item
In many situations it is enough to consider continuous functions $u$. But it will be useful at some point to consider $\LP g$ with $g$ given by \eqref{eq-29} and it is not clear a priori that $g$ is continuous in the general case, see Theorem \ref{res-09} and Lemma \ref{res-14} below for instance. This is the reason why we do not restrict $\dom\LP$ to continuous functions.
\item
The notation $v=\mathcal{L}u$ almost rightly suggests that $v$ is
a function of $u.$ Indeed, when $u$ is in $\dom\LP ,$
the Doob-Meyer decomposition of the special semimartingale
$u(t,X_t)$  into its predictable bounded variation part $\int
v_s\,ds$ and its local martingale part is unique. But one can
modify $v=\LP u$ on a small (zero-potential) set without
breaking the martingale property. As a consequence, $u\mapsto
\LP u$ is a multivalued operator and $u\mapsto
\LP u$ is an almost linear operation.
\item

 Suppose that $t_o$ is a fixed time of discontinuity of $P,$ i.e.\ $P(X_{t_o}\not=\lg X{t_o})>0.$ Then, in general a continuous function $u$ cannot be in $\dom\LP$. For this reason, one should think of the notion of extended generator for processes $P$ that do not have any fixed time of discontinuity:
 $P(X_t\not=\lg Xt)=0,$ for all $t\in\ii.$
\end{enumerate}\end{remarks}

The notion of generator is tightly connected with that of martingale problem.

\begin{definition}[Martingale problem]\label{def-02}
Let $\mathcal{C}$ be a class of measurable real functions $u$ on $\iX$
and for each $u\in\mathcal{C},$ let $\mathcal{L}u:\iX\to\RR$  be a
measurable function such that $\II01 |\mathcal{L}u(t,\omega_t)|\,dt<\infty$ for all $\omega\in\OO.$ Take also a probability measure
$\mu_0\in\PX.$ One says that $Q\in\PO$ is a solution to the
martingale problem $\MP(\mathcal{L},\mathcal{C};\mu_0)$ if
$Q_0=\mu_0\in\PX$ and for all $u\in\mathcal{C},$ the process
$$
    u(t,X_t)-u(0,X_0)-\II0t \mathcal{L}u(s,X_s)\,ds
$$
is  a local $Q$-martingale.
\end{definition}
As in Definition \ref{def-05}, this local martingale admits a  \cadlag\ $Q$-version.
\\
Playing with the definitions, it is clear that any Markov
law $Q\in\PO$  is a
solution to $\MP(\mathcal{L}^Q,\mathcal{C};Q_0)$ where $\mathcal{L}^Q$ is the extended generator of $Q$ and
$\mathcal{C}$ is any nonempty subset of $\dom\mathcal{L}^Q.$

Our aim is to show that the extended generator can be computed by
means of a stochastic derivative.

\begin{definition}[Integration time]
Let $u$ be a measurable real function on $\iX$ and $\tau$ be a stopping
time. We say that $\tau$ is a $P$-integration time of $u$ if the
family of random variables $\{u(\Xb^\tau_t);t\in\ii\}$ is
uniformly $P$-integrable.
\end{definition}

\begin{definition}[Stochastic derivative of a Markov process]\label{def-04}
Let $P$ be a Markov process and $u$ be a measurable real function on $\iX$. We say
that $u$ admits a stochastic derivative  under $P$ at time
$t\in\ii$ if for $P_t$-almost all $x\in\XX$ there exists a
$P$-integration time $\sigma^x$ of $u$  such that $\sigma^x\ge t,$
$P\ae$ and for any $P$-integration time $\tau$ of $u$ satisfying
$\tau>\sigma^x,$ $P\ae$ the following limit
$$
   L^Pu(t,x):=
    \Limh E_P\left(\frac1h [u(\Xb^\tau_{t+h})-u(t,x)]
    \mid  X_t=x \right)
$$
exists and does not depend on $\tau.$
\\
If $u$ admits a stochastic derivative for $dtP_t(dx)$-almost all
$(t,x),$ we say that $u$ belongs to the domain $\dom L^P $ of
the stochastic derivative $L^P$ of the Markov process $P.$
\\
If the function $u$ does not depend on the time variable $t,$ we
denote
$$
    L^P_tu(x)=L^Pu(t,x).
$$
\end{definition}

This extension of Nelson's definition by means of integration times seems to be new. It is consistent since the supremum of two
integration times is still an integration time. Indeed, the
supremum of two stopping times is a stopping time and for all $t,$
$|u(\Xb^{\tau\vee\tau'}_t)|\le
|u(\Xb^{\tau}_t)|+|u(\Xb^{\tau'}_t)|.$
\\ 
As in Definition \ref{def-05}, we do not restrict the domain of the stochastic derivative to continuus functions, see Remark \ref{rem-02}-(d).
\\
Since $P$ is a Markov process, we have also
$$
   L^Pu(t,x)=
    \Limh E_P\left(\frac1h [u(\Xb^\tau_{t+h})-u(t,x)]
    \mid \tau>t, X_t=x \right).
$$
We denote $\Pb$ the product of the Lebesgue measure on $\ii$ by
the process $P:$ $\Pb(dtd\omega)=dtP(d\omega).$ In the sequel,
we shall be concerned with the function space
$L^p(\ii\times\OO,\Pb).$

\begin{lemma}\label{res-21}
For all $h>0,$ let
$k^h\ge0$ be a measurable convolution kernel such that $\supp
k^h\subset[-h,h]$ and $\int_{\mathbb{R}}k^h(s)\,ds=1.$
\\
Let $P$ be a bounded positive measure on $\OO$ (which may not be a probability measure) and  $v(t,\omega)$ be a function in $L^p(\ii\times\OO,\Pb)$ with
$1\le p<\infty.$ Define for all $h>0$ and $t\in\ii,$ $k^h *
v(t)=\int_{\mathbb{R}}k^h(t-s)v_s\,ds$ where $v$ is extended by
putting $v_s=0$ for all $s\not\in\ii.$
\\
Then, $k^h *v$ is in $L^p(\ii\times\OO,\Pb)$ and
    $
    \Limh k^h *v=v\  \textrm{in }L^p(\ii\times\OO,\Pb).
    $
\end{lemma}
We see that $k^h (s)\, ds$ is a probability measure on $\RR$ which
converges narrowly to the Dirac measure $\delta_0$ as $h$ tends down to zero.

\begin{proof}
In this lemma, we endow as usual $\OO$ with the Skorokhod topology which turns it into a Polish space and has the interesting property that its Borel $\sigma$-field matches with the cylindrical $\sigma$-field.

We denote $L^p(\ii\times\OO,\Pb)=L^p(\Pb)$ and start the proof by
showing that $k^h *v\in L^p(\Pb).$ For $P$-almost all $\omega,$
$v(\cdot,\omega)\in L^p(\ii)$ so that $k^h*v(\cdot,\omega)$ is
also in $L^p(\ii)$ with $\|k^h*v(\cdot,\omega)\|_{L^p(\ii)}\le
\|v(\cdot,\omega)\|_{L^p(\ii)}.$ It remains to integrate with
respect to  $P(d\omega)$ to obtain
\begin{equation}\label{eq-42}
    \|k^h*v\|_{L^p(\Pb)}\le \|v\|_{L^p(\Pb)}<\infty.
\end{equation}
Now, we prove the convergence. As $p$ is finite, the space
$C_c(\ii\times\OO)$ of all  continuous functions with a compact
support in $\ii\times\OO$ is dense in $L^p(\Pb).$ We approximate
$v$ in $L^p(\Pb)$ by a sequence $\seq vn$ in $C_c(\ii\times\OO).$
For all $h$ and $n$
\begin{eqnarray*}
  \|k^h *v-v\|_{L^p(\Pb)}
  &\le& \|k^h *(v-v_n)\|_{L^p(\Pb)} +\|k^h *v_n-v_n\|_{L^p(\Pb)}
    +\|v_n-v\|_{L^p(\Pb)}  \\
   &\le& \|k^h *v_n-v_n\|_{L^p(\Pb)}+2 \|v-v_n\|_{L^p(\Pb)}
\end{eqnarray*}
where we used \eqref{eq-42}.
\\
Take an arbitrary small $\eta>0$ and choose $n$ large enough for
$\|v-v_n\|_{L^p(\Pb)}\le\eta$ to hold. Then,
\begin{equation}\label{eq-01}
     \|k^h *v-v\|_{L^p(\Pb)}\le
     \|k^h *v_n-v_n\|_{L^p(\Pb)}+2\eta.
\end{equation}
Fix this $n.$ Since $v_n$ is in $C_c(\ii\times\OO),$ it is a
uniformly continuous function. Therefore, for all $\eta>0,$ there
exists $h(\eta)>0$ such that for any $t,t',\omega,\omega'$
satisfying $|t-t'|+d_\OO(\omega,\omega')\le h(\eta),$ we have
$|v_n(t',\omega')-v_n(t,\omega)|\le\eta,$ where $d_\OO$ is the
Skorokhod metric on $\OO.$ In particular, with $\omega=\omega',$
we see that
\begin{equation*}
    |t'-t|\le h(\eta) \Rightarrow
    \sup_{\omega\in\OO}|v_n(t',\omega)-v_n(t,\omega)|\le\eta.
\end{equation*}
Because of the property: $\supp k^h\subset[-h,h],$ we deduce from
this that for any $\omega\in\OO,$
$|k^h*v_n(t)-v_n(t)|\le\int_{\mathbb{R}}|v_n(t-s)-v_n(t)|k^h(s)\,ds\le\eta$
as soon as $h\le h(\eta)/2.$ Consequently $ \|k^h
*v_n-v_n\|_{L^p(\Pb)}\le P(\OO)\eta.$ Finally, with \eqref{eq-01}
this leads us to $\|k^h *v-v\|_{L^p(\Pb)}\le (2+P(\OO))\eta.$
Since $\eta$ is arbitrary, this shows that
    $
    \lim_{h\rightarrow0}\|k^h
    *v-v\|_{L^p(\Pb)}=0,
    $
which is the desired result
\end{proof}

\begin{proposition}\label{res-01} Let $P$ be a Markov process and $u$ be a function in the domain $\dom\LP $ of the
extended generator $\LP $ of $P.$ We suppose in addition
that there exists  $1\le p<\infty$ such that $E_P\II01
|\LP u(t, X_t)|^p\,dt<\infty.$ Then,
\begin{equation}\label{eq-43}
    \Limh E_P\II0{1-h}\left|
    \frac 1hE_P[u(t+h,X_{t+h})-u(t,X_t)\mid X_t]-\LP u(t,X_t)\right|^p\,dt=0.
\end{equation}
\end{proposition}

\begin{proof} We denote
$v_t=\LP u(t,X_t).$ Choosing the specific convolution
kernel $k^h=\frac 1h \1_{[-h,0]},$ and relying on the very
definition of the extended generator, we obtain
\begin{multline*}
\frac 1hE_P[u(t+h,X_{t+h})-u(t,X_t)\mid X_t]=  \frac 1hE_P[u(t+h,X_{t+h})-u(t,X_t)\mid \XXX0t]\\=E_P[k^h*v(t)\mid \XXX0t]=E_P[k^h*v(t)\mid X_t].
\end{multline*}
On the other hand, by Jensen's inequality and Fubini's theorem
\begin{eqnarray*}
  E_P\II01\Big|E_P[k^h*v(t)\mid X_t]-v_t\Big|^p\,dt
  &=&  E_P\II01\Big|E_P[k^h*v(t)-v_t\mid X_t]\Big|^p\,dt \\
  &\le&  E_P\II01|k^h*v(t)-v_t|^p\,dt.
\end{eqnarray*}
Our hypothesis $v\in L^p(\ii \times\OO,\overline{P})$ is precisely
the assumption of previous Lemma \ref{res-21} which insures that
$\Limh E_P\II01|k^h*v(t)-v_t|^p\,dt=0.$ Gathering these
considerations, we obtain \eqref{eq-43}.
\end{proof}

A variant of this proposition already appears in \cite{Foe86}. But it seems to the author that its proof is incomplete and that it is difficult to avoid a convolution argument such as Lemma \ref{res-21}.

\begin{theorem}\label{res-02}
Let $P$ be a Markov process and $u$ be a function in the domain
$\dom\LP $ of the extended generator $\LP $ of
$P.$ Then, $u$ belongs to $\dom L^P $ and
    $$
      \LP u
   =L^Pu,\quad dtP_t(dx)\ae
    $$
\end{theorem}

\begin{proof}
By the definition of the extended generator, there exists a
localizing sequence $\seq{\tau}k$ of stopping times, i.e.\ such
that $\Lim k \tau_k=\infty,$ $P\ae$ and for all $k\ge1,$ the
stopped process $M^{\tau_k}$ where
$$
    M_t=u(t,X_t)-\II0{t} \LP u(s,X_s)\,ds,
$$
is a uniformly integrable martingale.  By considering the sequence of stopping
times  $\inf\{t\in\ii; \II0t |\LP u(s,X_s)|\,ds\ge
k\}\in\ii\cup\{\infty\}$ indexed by $k\ge1,$ it is easy to show that
$\seq{\tau}k$ can also be chosen such that for each $k,$ $\tau_k$
is also an integration time of $u.$

Let us consider a fixed integration time $\tau$ of $u$ such that
$M^\tau$ is a uniformly integrable martingale. Denoting
$v^\tau(t)=\1_{\{t\le\tau\}}\LP u(t,X_t)$ and choosing
$k^h=\frac1h\1_{[-h,0]}$ as in the proof of Proposition
\ref{res-01},  we see that $\frac
1h[u(\Xb^\tau_{t+h})-u(\Xb^\tau_t)]-k^h*v^\tau(t)$ is a
martingale. It follows that
\begin{equation*}
  \frac 1hE_P[u(\Xb^\tau_{t+h})-u(t,X^\tau_t)\mid \XXX0t]
    =E_P[k^h*v^\tau(t)\mid  \XXX0t] 
   = \1_{\{t\le\tau\}}E_P[k^h*v(t)\mid X_t].
\end{equation*}
Remark for future use that this implies that
\begin{equation}\label{eq-44}
     \frac 1hE_P[u(\Xb^\tau_{t+h})-u(t, X^\tau_t) \mid \XXX0t]=\1_{\{t\le\tau\}}
      \frac 1hE_P[u(\Xb^\tau_{t+h})-u(t, X^\tau_t)\mid 
      X^\tau_t].
\end{equation}
Then, as for \eqref{eq-43} with $p=1,$ we obtain
\begin{equation*}
    \Limh  E_P\int_{[0,\tau\wedge (1-h)]}\left|
    \frac 1hE_P[u(\Xb^\tau_{t+h})-u(t, X^\tau_t)\mid  \XXX0t]-\LP u(t,
    X_t)\right|\,dt=0
\end{equation*}
and with Fatou's lemma
\begin{equation*}
      E_P\II0{1-h}\liminf_{h\downarrow0}\1_{\{t\le\tau\}}\left|
    \frac 1hE_P[u(\Xb^\tau_{t+h})-u(t, X^\tau_t)\mid  \XXX0t]-\LP u(t,
    X_t)\right|\,dt=0.
\end{equation*}
But, since $u$ is in $\dom\LP $, $\Limh \frac
1hE_P[u(\Xb^\tau_{t+h})-u(t, X^\tau_t)\mid \XXX0t]$ appears as
the computation of the derivative of an absolutely continuous
function. Therefore, this limit exists for Lebesgue-almost all $t$
and the $\liminf_{h\downarrow0}$ arising from the application of
Fatou's lemma is a  genuine limit\footnote{The absolute continuity
plays a crucial role. Note that it is also of primary importance in the
definition of the extended generator.}. Hence,
\begin{equation*}
      E_P\II0{1-h}\1_{\{t\le\tau\}}\Limh\left|
    \frac 1hE_P[u(\Xb^\tau_{t+h})-u(t, {X}_t)\mid \XXX0t]-\LP u(t,
    X_t)\right|\,dt=0
\end{equation*}
and with \eqref{eq-44} this shows us that for
$\overline{P}$-almost all $(t,\omega)$ we have
\begin{equation*}
   \1_{\{\tau(\omega)\ge t\}} \Limh\frac 1{h}
    E_P\left[u(\Xb^\tau_{t+h})-u(t,X_t(\omega))\mid  X^\tau_t\right](\omega)
    = \1_{\{\tau(\omega)\ge t\}} \LP u(t,X_t(\omega)).
\end{equation*}
As the left-hand side vanishes when $\tau(\omega)=t,$ we obtain
\begin{equation*}
   \1_{\{\tau> t\}} \Limh\frac 1{h}
    E_P\left[u(\Xb^\tau_{t+h})-u(t,X_t)\mid \tau>t, X_t\right]
    = \1_{\{\tau> t\}} \LP u(t,X_t).
\end{equation*}
This results holds true for any integration time $\tau$ of $u$
such that  $M^\tau$ is a  uniformly integrable martingale.

By assumption, for $\overline{P}$-almost all $(t,\omega)$ there
exists $k(t,\omega)$ large enough for the localizing time
$\tau_{k(t,\omega)}$ to satisfy $\tau_{k(t,\omega)}(\omega)\ge t.$
Choosing $\sigma^{X_t(\omega)}=\tau_{k(t,\omega)},$ we obtain
for $dtP_{t}(dx)$-almost all $(t,x)$ an integration time
$\sigma^x\ge t$ such that any integration time $\tau>\sigma^x$
satisfies
\begin{equation*}
    \Limh\frac 1{h}
    E_P\left[u(\Xb^\tau_{t+h})-u(t,x)\mid  X_t=x\right](\omega)
    = \LP u(t,x).
\end{equation*}
This completes the proof of the
theorem.
\end{proof}

Let us investigate a partial converse of Theorem \ref{res-02}.

\begin{proposition}\label{res-03}
Let $P$ be a Markov process, $u$ and  $v$ be  measurable real functions on $\iX$
which satisfy the following requirements. The function $v$
verifies $\II01 \left|v(t,X_t)\right|\,dt<\infty,$ $P\as$  and
there exists a sequence $\Seq\tau k$ of integration times of $u$
such that
$\Lim k\tau_k=\infty,$ $P\as$ and for each $k\ge1,$
\begin{equation}\label{eq-03}
    \Limh E_P\II0{1-h}\left|\frac1h E_P[u(\Xb^{\tau_k}_{t+h})-u(\Xb^{\tau_k}_{t})\mid X_t]-\1_{\{t\le\tau_k\}}v(t,
    X_t)\right|\,dt =0
\end{equation}
 Then, $u$ belongs to $\dom\LP $ and $\dom L^P $ and
    $$
      \LP u=L^Pu
   =v,\quad dtP_t(dx)\ae
    $$
\end{proposition}

Note that if $P$ admits a fixed time of discontinuity, there might be many continuous  functions $u$ which do not verify \eqref{eq-03}.

\begin{proof}
The proof relies on the subsequent easy analytic result.

\noindent \textsf{Claim}.\ \textit{Let $a,b$ be two measurable
functions on $\ii$ such that $a$ is right continuous, $b$ is
Lebesgue-integrable and
    $
    \Limh \int_{[0,1-h]}\left|\frac1h\{a(t+h)-a(t)\}-b(t)\right|\,dt=0.
    $
Then, $a$ is absolutely continuous and its distributional
derivative is $\dot a=b.$}
\\
To see this, remark first that $t\mapsto\1_{\{0\le t\le
1-h\}}\frac1h\{a(t+h)-a(t)\}$ is integrable for any 
$0<h\le 1.$ Take any $0\le r\le s<1.$ On one hand, we have
    $\Limh \II rs  \frac1h\{a(t+h)-a(t)\}\,dt=\II rs b(t)\,dt$
and on the other one:
    $\II rs  \frac1h\{a(t+h)-a(t)\}\,dt=\frac1h \II r{r+h}a(t)\,dt -\frac1h \II s{s+h}a(t)\,dt,$
so that with the assumed right continuity of $a$ we have
     $\Limh \II rs  \frac1h\{a(t+h)-a(t)\}\,dt=a(s)-a(r).$
Therefore $a(s)-a(r)=\II rs b(t)\,dt$ which is the claimed
property.

Let us fix $\tau^k$ as in the assumption of the proposition. We
write $E=E_P,$ $u_t=u(\Xb^{\tau_k}_t)$ and
$v_t=\1_{\{t\le\tau_k\}}v(\Xb^{\tau_k}_t)$ to simplify the notation.
Define the family of stopping times $\sigma_k:=\inf\{s\in\ii;\II0s
|v(t,X_t)|\,dt\ge k\}$ where $k$ describes the integers. By
considering the stopping times $\sigma_k\wedge\tau_k,$ we can
assume without loss of generality that $$v\in L^1(\Pb).$$ Fix $0\le
r<1.$ We have
\begin{multline*}
    \left|E\left[\II r{1-h}\left(\frac1h\{u_{t+h}-u_t \}-v_t\right)\,dt\mid X_r\right]\right| \\
   \le  E\left[\II r{1-h}E\left(\big|\frac1h\{u_{t+h}-u_t \}-v_t\big|\mid X_t\right)\,dt\mid
   X_r\right]
\end{multline*}
With \eqref{eq-03} and Fatou's lemma, we obtain
\begin{multline*}
  E\left(\liminf_{h\downarrow0} \left|E\left[\II r{1-h}\left(\frac1h\{u_{t+h}-u_t \}-v_t\right)\,dt\mid X_r\right]\right|\right) \\
   \le \Limh  E \II r{1-h}E\left(\big|\frac1h\{u_{t+h}-u_t \}-v_t\big|\mid X_t\right)\,dt=0.
\end{multline*}
Hence, there exists a sequence $\seq hn$ of positive numbers such
that $\Lim n h_n=0$ and
\begin{equation*}
    \Lim n \II r{1-h_n}E\left[\left(\frac1{h_n}\{u_{t+h_n}-u_t \}-v_t\right)\mid
   X_r\right]\,dt=0,\quad P\ae
\end{equation*}
It remains to apply the result of the above claim to
    $a(t)=E\left[u_t\mid X_r\right]$ and
    $b(t)=E\left[v_t\mid X_r\right]$
to see that for all $0\le r\le s<1,$
    $E\left[u_s-u_r-\II rs v_t\,dt\mid\XXX0r\right]=0.$
This proves that $M^{\tau_k}$ is a $P$-martingale where
$$M_s:=u(s,X_s)-u(0,X_0)-\II0s v(t,X_t)\,dt.$$
With the assumptions that $\Lim k\tau_k=\infty,$ $P\as$,
$\1_{[0,\tau_k]}v\in L^1(\Pb)$  and the fact that 
 $\{u(\Xb^{\tau_k}_t);t\in\ii\}$ is uniformly $P$-integrable by the very definition of the integration time $\tau_k$,
we conclude  that $M$ is a local $P$-martingale. Therefore, $u$
belongs to $\dom\LP$ and $\LP u=v.$ And we
also have $u\in\dom L^P$ and  $L^Pu=\LP u$ by Theorem
\ref{res-02}.
\end{proof}

\section{Feynman-Kac processes}\label{sec-FK}

Let $R$ be a probability measure on $\OO$ which is a stationary Markov process with the invariant
\emph{probability}  measure
\begin{equation}\label{eq-06}
     m:=R_t\in\PX,\quad\forall t\in\ii.
\end{equation}
We also consider a lower bounded potential $V,$ i.e.\ a measurable function
$V:\iX\to \RR$ such that
\begin{equation}\label{eq-04}
    \inf_{\iX}V\ge -\lo
\end{equation}
with $0\le \lo<\infty.$ Let $\gg$ be a  nonnegative $m$-integrable
function on $\XX.$ 
In this section we look at the real valued process
\begin{equation}\label{eq-11}
    G_t:=E_R\left[\exp\left(-\II t1 V_s(X_s)\,ds\right)\gg(X_1)\mid
    \XXX0t\right]=:g_t(X_t),
    \quad t\in\ii ,\quad \gg\ge0
\end{equation}
which we call a Feynman-Kac process. Last equality, where
$g_t:\XX\to[0,\infty)$ is a measurable function, is a consequence
of the Markov property of $R$. 

\subsection*{Orlicz spaces}
The mere integrability of $\gg$ is sufficient for defining $G,$ but it will not be enough in general for our purpose. We are going to assume that $\gg$ is in some Orlicz space 
$$
\Lg:=\left\{u:\XX\to\RR; \textrm{ measurable, } \IX \gamma(a_o|u|)\,dm<\infty,\textrm{ for some } a_o>0 \right\}
$$
associated with the Young function $\gamma.$ Recall that $\gamma:\RR\to[0,\infty]$ is a Young function if it is convex, even, \lsc\ and $\gamma(0)=0.$ Important instances are 
\begin{enumerate}[-]
\item
$\gamma(a)=\gamma_p(a):= |a|^p/p,$ with $1\le p<\infty,$ then  $L^{\gamma_p}(m)=L^p(m);$
\item
$\gamma(a)=\gamma_\infty(a):=\left\{\begin{array}{rl}
0& \textrm{if }|a|\le1\\ \infty& \textrm{otherwise}
\end{array}\right.,$ then $L^{\gamma_\infty}=L^\infty(m).$
\end{enumerate}
Let us introduce the functions 
\begin{eqnarray}\label{eq-45}
\theta(a)&:=&e^a -a-1,\quad a\in\RR,\\ 
\theta^*(a)&:=&(a+1)\log(a+1)-a,\quad a\in[-1,\infty)\nonumber
\end{eqnarray}
with the convention $0\log 0=0.$ They are convex conjugate to each other and $\theta(a)=\log \mathbb{E}e^{a(N-1)}$ where $N$ is a Poisson(1) random variable. Moreover, $\theta(|a|)$ and $\theta^*(|b|)$ are Young functions which are also convex conjugate to each other.
\\
Two other important Orlicz spaces are 
\begin{enumerate}[-]
\item
$\gamma(a)=\theta(|a|)$ corresponds to the following $L^\gamma:$
$$
\Lexp:=\left\{u:\XX\to\RR; \textrm{ measurable, } \IX e^{a_o|u|}\,dm<\infty, \textrm{ for some } a_o>0 \right\},
$$
\item
$\gamma(a)=\theta^*(|a|)$ corresponds to the following $L^\gamma:$
$$
\LlogL:=\left\{u:\XX\to\RR; \textrm{ measurable, } \IX |u|\log_+|u|\,dm<\infty \right\},
$$
\end{enumerate}
where we use the assumed boundedness of the positive measure $m$ in  the above expressions.
\\
The Luxemburg norm of $\Lg$ is defined by $\|u\|_{\Lg}:=\inf_{}\{\alpha>0; \IX \gamma(|u|/\alpha)\,dm\le1\}.$ Let $\gamma^*(b):=\sup_{a\ge0}\{ab-\gamma(a)\}\in[0,\infty],$ $b\ge0,$ be the convex conjugate of $\gamma.$ It follows immediately from Fenchel's inequality
$
ab\le \gamma(a)+\gamma^*(b),
$
that the H\"older inequality 
$$
\|uv\|_{L^1(m)}\le 2 \|u\|_{\Lg}\|v\|_\Lgs,\quad u\in\Lg,v\in\Lgs
$$
holds true. In particular, since $\theta(|\cdot|)$ and $\theta^*(|\cdot|)$ are convex conjugate to each other, we have 
$
\|uv\|_{L^1(m)}\le 2 \|u\|_{\LlogL}\|v\|_\Lexp,$ for all $u\in\LlogL,v\in\Lexp.
$
\\
The Young function $\gamma$ is said to satisfy the \emph{condition}  $\Delta_2$ if there exist  constants $C,A>0$ such that $\gamma(2a)\le C\gamma(a),$ for all $a\ge A.$ 
The spaces $\LlogL$ and  $L^p(m)$ with $1\le p< \infty$ satisfy $\Delta_2.$ But $L^\infty(m)$ and $\Lexp$ do not.

\subsection*{Preliminary results}

We assume that the next finite entropy condition is satisfied
$$
\gg\ge0,\quad \IX \gg\log_+\gg\,dm<\infty
$$ 
and we pick a Young function $\gamma$ such that
\begin{equation}\label{eq-05} 
     \IX \gamma(\gg)\,dm<\infty  \quad\textrm{ and} \quad
\LlogL\subset\Lg\subset L^p(m) \textrm{ for some }  1<p<\infty.
\end{equation}
In particular, we have $\gamma\in \Delta_2.$
\\
Because of \eqref{eq-06}, \eqref{eq-04} and \eqref{eq-05}, with $G_t$ given by \eqref{eq-11}, we have for all $t\in\ii$ and $\alpha>0,$
$$
 	\IX \gamma(g_t/\alpha)\,dm
	=E_R \gamma(G_t/\alpha)
\le E_R\gamma(e^{\lo}G_1/\alpha)
\le C_{\gamma,\lo} E_R (G_1/\alpha)
$$
where $C_{\gamma,\lo}>0$ is some finite constant which can be derived by means of the condition $\Delta_2.$ Optimizing in $\alpha$ leads us to
$$
\|g_t\|_\Lg\le C_{\gamma,\lo}\|\gg\|_\Lg,\quad \forall t\in\ii.
$$
Recall that a real valued process $G$ is said to admit a \cadlag\
version if there exists a modification $G'$ of $G,$ i.e.\
$R(G_t\not =G'_t)=0$ for all $t\in\ii,$ with its sample paths in
$\DD:=D(\ii,\RR).$

\begin{lemma}\label{res-04}
Let us assume that in addition to \eqref{eq-06}, \eqref{eq-04} and \eqref{eq-05}, we have
\begin{equation}\label{eq-08}
    \II01 \|V_t\|_{L^1(m)}\,dt<\infty.
\end{equation}
 Then, the process $G$ admits a \cadlag\
version. In the sequel $G$ will always be assumed to be this
$\DD$-valued version.
\\
It is a nonnegative semimartingale which satisfies the so-called Feynman-Kac semigroup property:
\begin{equation}\label{eq-07}
    E_R\left[\exp\left(-\II st V_r(X_r)\,dr\right)G_t\mid
    \XXX0s\right]=G_s,
    \quad 0\le s\le t\le1.
\end{equation}
Moreover, denoting $G_*:=\sup_{t\in\ii}G_t,$ we have
\begin{equation*}
   \|G_*\NgR\le C_{\gamma,\lo} \|\gg\Ngm
\end{equation*}
for some finite positive constant $C_{\gamma,\lo}$.\\ This implies that $\{\gamma(G_t);t\in \ii\}$ is uniformly integrable in
$L^1(R).$
\end{lemma}

\proof Let us prove \eqref{eq-07}. For all $0\le s\le t\le1,$
\begin{eqnarray*}
   && E_R\left[\exp\left(-\II st V_r(X_r)\,dr\right)G_t\mid
   \XXX0s\right]\\
    &=&
  E_R\left[\exp\left(-\II st V_r(X_r)\,dr\right)E_R\left\{\exp\left(-\II t1 V_r(X_r)\,dr\right)G_1\mid \XXX0t\right\}\mid \XXX0s\right]\\
  &=& E_R\left[E_R\left\{\exp\left(-\II s1 V_r(X_r)\,dr\right)G_1\mid \XXX0t\right\}\mid \XXX0s\right] \\
  &=& E_R\left[\exp\left(-\II s1 V_r(X_r)\,dr\right)G_1\mid \XXX0s\right] \\
  &=& G_s
\end{eqnarray*}
which is \eqref{eq-07}.
\\
Let us define $\Vt:=V+\lo\ge0$ and for all $t\in\ii$
$$
    \Gt_t:=e^{-\lo(1-t)}G_t=E_R\left[\exp\left(-\II t1\Vt_s(X_s)\,ds\right)\Gt_1\mid\XXX0t
    \right]
$$
where $\Gt_1=G_1=\gg(X_1).$ Because $\Vt\ge0$, we see that  for all
$0\le s\le t\le1,$
\begin{eqnarray*}
  E_R\left(\Gt_t\mid\XXX0s\right)
  &=& E_R\left[\exp\left(-\II t1\Vt_r(X_r)\,dr\right)\Gt_1\mid\XXX0s\right] \\
  &\ge& E_R\left[\exp\left(-\II
  s1\Vt_r(X_r)\,dr\right)\Gt_1\mid\XXX0s\right]\\
  &=&\Gt_s.
\end{eqnarray*}
In other words, $\Gt$ is a nonnegative submartingale.
\\
It follows from the fact that the forward filtration satisfies the
standard assumptions and from a well-known result of the general
theory of stochastic processes that $\Gt$ admits a \cadlag\
modification (still denoted by $\Gt$) if
    $
t\in\ii\mapsto E_R\Gt_t\in[0,\infty)
    $
is a right continuous real function. But this latter property is a
direct consequence of Lebesgue's dominated convergence theorem and
the pathwise right continuity of
$$
    t\in\ii\mapsto \exp\left(-\II t1\Vt_r(X_r)\,dr\right)\in (0,1]
$$
which is satisfied under the assumption \eqref{eq-08}: $E_R\II01
|V_t(X_t)|\,dt<\infty,$ which implies that $\II01
|\Vt_t(X_t)|\,dt<\infty,$ $R\as$\footnote{Remark that without the
assumption that $\II01 |\Vt_t(X_t)|\,dt<\infty,$ $R\as$ and with
the convention $e^{-\infty}=0,$ $ t\in\ii\mapsto \exp\left(-\II
t1\Vt_r(X_r)\,dr\right)\in [0,1]$ is well-defined $R\as$, but it
might fail to be right continuous.}
\\
Furthermore,  we have
$E_R \gamma(\Gt_t)\le E_R \gamma(\Gt_1)<\infty$ by Jensen's inequality and the submartingale property.  Doob's maximal inequality, which holds for any nonnegative submartingale and any Young function $\gamma$ which verifies \eqref{eq-05}\footnote{For Doob's inequality in the class $L\log L,$ see \cite[p. 54]{RY99} for instance.}, tells us that there exists a positive finite constant $c_\gamma<\infty$ such that
\begin{equation*}
    \|\sup_{t\in\ii}\gamma(\Gt_t)\|_{L^1(R)}
    \le c_\gamma  \sup_{t\in\ii}\|\gamma(\Gt_t)\|_{L^1(R)}
    = c_\gamma\|\gamma(\Gt_1)\|_{L^1(R)}
    =c_\gamma\|\gamma(\gg)\|_{L^1(m)}<\infty.
\end{equation*}
Hence $\{\gamma(\Gt_t);t\in\ii\}$ is uniformly integrable in $L^1(R).$
Since the product of two semimartingales is still a
semimartingale, we  deduce that $G_t=e^{\lo(1-t)}\Gt_t$ is a
\cadlag\ semimartingale such that $\{\gamma( G_t);t\in\ii\}$ is uniformly
integrable in $L^1(R).$ This completes the proof of the lemma.
\endproof

Recall that since $R$ is a bounded nonnegative measure, a family
$\{H_t;t\in\ii\}$ of real valued measurable functions is uniformly
integrable in $L^1(R)$ if and only if there exists an increasing
convex function $\xi:[0,\infty)\to[0,\infty)$ such that $\Lim
a\xi(a)/a=+\infty$ and $\sup_{t\in\ii}E_R\xi(|H_t|)<\infty.$

\noindent \textsf{Claim}.\ \textit{ Let $A_t, B_t,$ $t\in\ii$ be two 
random variables such that both $\{\gamma( A_t);t\in\ii\}$ and
$\{\gamma^*( B_t);t\in\ii\}$ are uniformly integrable in $L^1(R)$. Then,
the family of products $\{A_tB_t;t\in\ii\}$ is uniformly
integrable in $L^1(R)$.}
\\
Let us prove this claim.  By hypothesis there exist two functions $\xi_1$ and
$\xi_2$ as above such that 
 $\sup_tE\xi_1(\gamma( A_t))<\infty$ and
$\sup_tE\xi_2(\gamma^* (B_t))<\infty$ where we wrote
$\sup_t=\sup_{t\in\ii}$ and $E=E_R$ for short.
  Let $\xi$ be the
convex envelope of $x\mapsto\xi_1(x/2)\wedge\xi_2(x/2).$ It is convex as a definition and still
increasing and satisfies $\Lim x\xi(x)/x=\infty.$ We also
obtain with Fenchel's inequality
    $\xi(|A_tB_t|)\le\xi(\gamma( A_t)+\gamma^*( B_t))
	\le \xi(2\gamma( A_t))/2+\xi(2\gamma^*( B_t))/2
    \le \xi_1(\gamma( A_t))+\xi_2(\gamma^*( B_t))$
for each $t\in\ii.$ Consequently,
    $\sup_tE\xi(|A_tB_t|)<\infty.$ This shows that $\{A_tB_t;t\in\ii\}$ is uniformly
integrable and completes the proof of the claim.

The assumption \eqref{eq-08} will not be strong enough for our
purpose. We strengthen it in the next lemma.

\begin{lemma}\label{res-05}
Let us assume in addition to \eqref{eq-06}, \eqref{eq-04} and
\eqref{eq-05} that the family $\{\gamma^*( V_t);t\in\ii\}$ is uniformly
integrable in $L^1(m)$. Then,
\begin{enumerate}
    \item $\{\frac1h\II t{t+h}|V_s|\,ds\ |G_{t+h}-G_t|;t\in\ii,h>0\}$ is uniformly
integrable in $L^1(R);$
    \item $\II01 G_tV_t(X_t)\,dt$ is in $L^1(R);$
    \item $\{V_tG_t(X_t);t\in\ii\}$ is uniformly
integrable in $L^1(R).$
\end{enumerate}
\end{lemma}

\proof We write $V_t=V_t(X_t),$ $\sup_t=\sup_{t\in\ii}$ and
$E=E_R$ for short.

\boulette{(1)} There exists a function $\xi$ as above such that
$\sup_tE\xi(\gamma^*(V_t))<\infty.$ But
    $E\xi\left(\frac1h \II t{t+h}\gamma^*(V_s)\,ds\right)
    \le \frac1h \II t{t+h}E\xi(\gamma^*(V_s))\,ds
    \le \sup_tE\xi(\gamma^*(V_t))<\infty.$
This shows that $\{\frac1h \II t{t+h}\gamma^*(V_s)\,ds;t\in\ii,h>0\}$ is
uniformly integrable. On the other hand, we already know by Lemma
\ref{res-04} that $\{\gamma(|G_{t+h}-G_t|);t\in\ii,h>0\}$ is also
uniformly integrable. The above claim permits us to conclude.

    \boulette{(2)} We see that
    $$E\II01 G_t|V_t|\,dt\le E(\gamma( G_*))+E\II01 \gamma^*(V_t)\,dt
    \le  E(\gamma( G_*))+\sup_t E\gamma^*(V_t)<\infty$$
which is finite by Lemma \ref{res-04} and the assumption that
$\{\gamma^*(V_t);t\in\ii\}$ is uniformly integrable.

    \boulette{(3)} The result directly follows from the above Claim, Lemma \ref{res-04} and our assumptions on $V.$
\endproof

\subsection*{The extended Feynman-Kac generator}

The main result of this section is the next theorem.

\begin{theorem}\label{res-06}
Let us take the following ingredients.
\begin{enumerate}[(i)]
\item
$R\in\PO$ is a stationary Markov process with invariant law $m=R_t\in\PX$ for all
$t\in\ii;$
\item
$\gamma$ is a Young function which satisfies \eqref{eq-05} and $\gamma^*$ is its convex conjugate;
\item
$V$ is a measurable function on $\iX$ which is
bounded below and is such that $\{\gamma^*( V_t);t\in\ii\}$ is uniformly
integrable in $L^1(m);$
\item
$\gg$ is a nonnegative function on $\XX$ 
in $\Lg.$
\end{enumerate}
Then, the function $g:(t,x)\in\iX\mapsto g_t(x)\in[0,\infty)$
which is defined for all $t\in\ii,$ $m$-almost everywhere by \eqref{eq-11}:
\begin{equation*}
    g_t(X_t):=E_R\left[\exp\left(-\II t1 V_s(X_s)\,ds\right)\gg(X_1)\mid
    \XXX0t\right],
    \quad R\as,
\end{equation*}
is a nonnegative  function in $L^\gamma(\iX,dtm(dx))$ which is in $\dom
L^R$ and in $\dom\LR.$ Moreover, it satisfies
$$
    L^Rg(t,x)=\LR g(t,x)=V_t(x)g_t(x),
    \quad dtm(dx)\ae
$$
and $\IiX |V_t(x)|g_t(x)\,dtm(dx)<\infty.$
\end{theorem}

\proof The proof is based on an application of Proposition
\ref{res-03} with $u(t,x)=g_t(x)$ and $v(t,x)=V_t(x)g_t(x).$ We
write $V_t=V_t(X_t)$ and $E=E_R$ for short.
\\
We know by Lemma \ref{res-05} that $\IiX
|V_t(x)|g_t(x)\,dtm(dx)=E\II01 |V_tG_t|\,dt<\infty.$ This implies
that $\II01 |V_tG_t|\,dt<\infty,$ $R\as$ We
have also seen at Lemma \ref{res-04} that $G_t$ is a right
continuous uniformly integrable process. It follows that we can
choose $\tau_k=\infty$ $R\as$ for all $k\ge1$ in formula
\eqref{eq-03} and that for all $0\le s\le t\le1,$ $t\in [s,1]\mapsto
E(G_t\mid X_s)$ is a right continuous real function. Therefore, to obtain
the announced results, it is sufficient to show that
\begin{equation*}
    \Limh E\II0{1-h}\left|\frac1h E_P[G_{t+h}-G_{t}\mid X_t]-V_tG_t\right|\,dt
    =0.
\end{equation*}
We decompose
$$
     -\frac1h E[G_{t+h}-G_{t}\mid  X_t]+V_tG_t=E_P[A^t_h+B^t_h+C^t_h\mid X_t]
$$
where
\begin{eqnarray*}
  A^t_h &:=& \frac1h \theta\left(-\II t{t+h}V_s\,ds\right)G_t \\
  B^t_h &:=& \frac1h \left(e^{-\II t{t+h}V_s\,ds}-1\right)(G_{t+h}-G_t) \\
  C^t_h &:=& G_t \frac1h \II t{t+h} (V_t-V_s)\,ds
\end{eqnarray*}
with $\theta(a):=e^a-a-1,$ $a\in\RR$ which we already met at \eqref{eq-45}. It remains to prove that
$$
    \Limh E\II0{1-h} |A^t_h|\,dt=\Limh E\II0{1-h} |B^t_h|\,dt=\Limh E\II0{1-h} |C^t_h|\,dt=0.
$$
\Boulette{$\Limh E\II0{1-h} |A^t_h|\,dt=0$} We have
\begin{multline*}
   0\le \frac1h\theta\left(-\II t{t+h}V_s\,ds\right) \\
    \le \left|\left(\frac1h\II t{t+h}V_s\,ds\right)\left(e^{-\II t{t+h}V_s\,ds}-1\right)\right|
    \le\lo(e^{\lo h}-1)+\frac1h\II t{t+h}|V_s|\,ds.
\end{multline*}
But $E \gamma^* \left(\frac1h\II t{t+h}|V_s|\,ds\right) \le
   \frac1h\II t{t+h}E \gamma^*( |V_s|)\,ds \le\sup_tE \gamma^*(| V_t|)<\infty$ by
   assumption.
Since $\Lim b \gamma^*(b)/b=\infty$ because $\gamma$ doesn't grow too fast, $\{A^t_h;t\in\ii,h>0\}$ is uniformly integrable. This leads
us to the desired convergence result since $\Limh A^t_h(\omega)=0$
for $dtR(d\omega)$-almost all $(t,\omega)\in \ii\times\OO.$

\Boulette{$\Limh E\II0{1-h} |B^t_h|\,dt=0$} Since $t\mapsto\II 0t
V_s\,ds$ is absolutely continuous and $t\mapsto G_t$ is right
continuous $R\as$, we see that
$$
    |B^t_h|\le e^{\lo h}\frac1h\II t{t+h}|V_s|\,ds\ |G_{t+h}-G_t|
    \underset{h\downarrow 0}\rightarrow 0,\quad R\as
$$
On the other hand we have shown at Lemma \ref{res-05} that
$\{B^t_h;t\in\ii,h>0\}$ is uniformly integrable.

\Boulette{$\Limh E\II0{1-h} |C^t_h|\,dt=0$} We have
    $$
  E\II0{1-h}  G_t \left|\frac1h \II t{t+h}
  (V_t-V_s)\,ds\right|\,dt
  \le E \left[G_*\II0{1-h}  \left|V_t-\frac1h \II t{t+h}V_s\,ds\right|\,dt\right]
    $$
where we put $V_t=0$ for all $t>1.$ By Lemma \ref{res-04}, $G_*\in
L^{\gamma}(R)$. Therefore the measure $G_*R$ is a  bounded measure and we can
apply Lemma \ref{res-21} with $v(t,\omega)=V_t(\omega)$ in
$L^1(\iO,G_*R)$ and $k^h=\frac 1h \1_{[-h,0]}.$ This completes the proof of the theorem.
\endproof

\section{Generalized $h$-transforms of a Markov process}\label{sec-Gh}

Let $R\in\PO$ be a stationary Markov process with the invariant probability
measure $m\in\PX$ as in Section \ref{sec-FK}. In the present section we consider the process
\begin{equation}\label{eq-10}
    P:= \ff(X_0)\exp\left(-\II01 V_t(X_t)\,dt\right)\gg(X_1)\, R\in\PO
\end{equation}
where $V:\iX\to\RR$ is a lower bounded measurable potential and
\begin{equation}\label{eq-09}
     \ff\in\Lgs,\gg\in \Lg,\quad \ff,\gg\ge0.
\end{equation}
It is assumed once for all that $$R(\ff(X_0)\gg(X_1)>0)>0$$ to discard the uninteresting trivial situation where $P=0.$ We normalize $\ff$ and $\gg$    to obtain $P(\OO)=1.$
\\
Remark that $\exp\left(-\II01 V_t(X_t)\,dt\right)$ is bounded. It
follows with the assumption \eqref{eq-09} that $\ff(X_0)\in L^{\gamma^*}(R),\gg(X_1)\in
L^\gamma(R)$ and that $\ff(X_0)\exp\left(-\II01
V_t(X_t)\,dt\right)\gg(X_1)$ is a nonnegative $R$-integrable
function. Hence, it can be normalized such that $P$ is a
probability measure.

\begin{definition}[Generalized $h$-transform of $R$]
Let $R\in\PO$ be a stationary Markov process which admits an invariant probability measure.
\\
A process $P\in\PO$ which is specified by formula
\eqref{eq-10}  is called a
generalized $h$-transform of $R,$ or a generalized $h$-process for short.
\end{definition}
It is not essential that $R$ is assumed to be a stationary Markov process in this definition.
\\
Our aim is to identify $P$ as the solution of a martingale
problem. To do it, we are going to derive the extended generator $\LP $ of
the generalized $h$-process $P$ on a  class of functions
$\mathcal{C}$ which is large enough to characterize $P$. With
Theorem \ref{res-02}, we see that we are on the way to compute its
stochastic derivative $L^P$ on
$\mathcal{C}.$

\subsection*{Playing with the Markov property}

\emph{Recall that $P\in\PO$ is a Markov process if and only if for all
$t\in\ii,$  $\XXX0t$ and $\XXX t1$ are
independent with respect to the conditional law $P(\cdot\mid
X_t).$} In other words, if and only if the past and future are
independent conditionally on the present. This property is
invariant with respect to time reversal. In particular the time
reversed process of $R$ is still Markov.  As with the definition
of $g$ at \eqref{eq-11}, one can define a measurable function
$f_t(x)$ on $\iX$ by the formula
\begin{equation}\label{eq-12}
    E_R\left[\ff(X_0)\exp\left(-\II 0t V_s(X_s)\,ds\right)\mid
    \XXX t1\right]=:f_t(X_t),\quad t\in\ii
    \quad R\as
\end{equation}
since  $E_R(a\mid\XXX t1)=E_R(a\mid X_t)$ for any
$\XXX0t$-measurable and integrable function $a.$ As
$\exp\left(-\II 0t V_s(X_s)\,ds\right)$ is bounded and $\ff(X_0)\in
L^\gs(R),$ we see that $f_t\in \Lgs$ for all $t\in\ii.$

\begin{proposition}\label{res-10}\
\begin{enumerate}
    \item The generalized $h$-process $P$ is Markov.
         \item For every $t\in\ii,$ $P_t\ll m$ and
\begin{equation}\label{eq-13}
    \frac{dP_t}{dm}=f_tg_t
\end{equation}
where $f_t$ and $g_t$ are  defined respectively by \eqref{eq-12}
and \eqref{eq-11} and  stand respectively  in $\Lgs$ and $\Lg$.
    \item  For every $0\le s\le t\le1,$
    \begin{eqnarray}
   \frac{dP_{[s,t]}}{dR_{[s,t]}}
    &=& \frac{dP_s}{dm}(X_s) g_s(X_s)^{-1}\exp\left(-\II st V_r(X_r)\,dr\right)g_t(X_t)\label{eq-14f}\\
    &=&  f_s(X_s) \exp\left(-\II st
    V_r(X_r)\,dr\right)f_t(X_t)^{-1}\frac{dP_t}{dm}(X_t)\label{eq-14b}\\
    &=& f_s(X_s)\exp\left(-\II st V_r(X_r)\,dr\right) g_t(X_t)
    \nonumber
\end{eqnarray}
where  no division by zero occurs in the sense that $g_s>0,$
$P_s\as$ and $f_t>0,$ $P_t\as$
\end{enumerate}
\end{proposition}

\proof \boulette{(1)}
 Fix $0<t<1$ and take two bounded nonnegative functions
$a$ and $b$ such that $a$ is $\XXX0t$-measurable and $b$ is $\XXX
t1$-measurable. Let us write $\alpha=\ff(X_0)\exp\left(-\II0t
V_s(X_s)\,ds\right)\in \sigma(\XXX0t)$ and $\beta=\exp\left(-\II t1
V_s(X_s)\,ds\right)\gg(X_1)\in \sigma(\XXX t1)$ so that $P=\alpha\beta\,R$
and
    $$
    E_P(ab\mid X_t)
    =\frac{E_R(ab\alpha\beta\mid X_t)}{E_R(\alpha\beta\mid X_t)}
    \overset{\checkmark}=\frac{E_R(a\alpha\mid X_t)E_R(b\beta\mid X_t)}{E_R(\alpha\mid X_t)E_R(\beta\mid X_t)}
    =E_P(a\mid X_t)E_P(b\mid X_t)
    $$
where we used the Markov property of $R$ at the marked equality.
This proves that $P$ is Markov.

\Boulette{(2) and (3)} As a general result of integration theory,
if $P=ZR$ with $Z\in L^1(R),$ then the push-forward
$P_\phi:=\phi\pf P$ of the measure $P$ by the measurable
application $\phi$ is absolutely continuous with respect to
$R_\phi:=\phi\pf R$ and
    $P_\phi=E_R(Z\mid \phi)\, R_\phi$
where $E_R(Z\mid \phi):=E_R(Z\mid \sigma(\phi))$ is the
conditional expectation of $Z$ with respect to the $\sigma$-field
$ \sigma(\phi)$  generated by $\phi.$ In particular, with
$\phi=\XXX st$ we obtain
$$
    P_{[s,t]}=E_R(dP/dR\mid\XXX st)\,R_{[s,t]}.
$$
We have
\begin{eqnarray*}
  &&E_R(dP/dR\mid\XXX st) \\
    &=& E_R\left[\ff(X_0)\exp\left(-\II01 V_r(X_r)\,dr\right)\gg(X_1)\mid\XXX st\right]\\
    &=& E_R\left[\ff(X_0)\exp\left(-\left\{\II0s +\II st+\II t1\right\}  V_r(X_r)\,dr\right)\gg(X_1)\mid\XXX
    st\right]\\
    &=& f_s(X_s)\exp\left(-\II st V_r(X_r)\,dr\right) g_t(X_t)
\end{eqnarray*}
where the Markov property of $R$ is used at last equality. In
particular, when $s=t$ this gives us \eqref{eq-13}. But with
\eqref{eq-13}, we see that for all $t,$ $f_t>0$ and $g_t>0,$
$P_t\as,$ $f_s(X_s)=g_s(X_s)^{-1}\frac{dP_s}{dm}(X_s)$ and
$g_t(X_t)=f_t(X_t)^{-1}\frac{dP_t}{dm}(X_t).$
\endproof

\subsection*{A preliminary result under a  finite entropy condition}
A seemingly innocent result is proved at Proposition \ref{res-13} below. But in fact it is a  mendatory technical key  to our approach. It states that, provided that the canonical process is a \emph{nice} $R$-semimartingale (see Definition \ref{def-01}), under the assumption that the \emph{relative entropy}  
$$H(P|R):=\int \log\dPR\,dP<\infty$$ is finite, if a large class of regular functions stands in $\dom\LR,$ then it is also in $\dom\LP.$

Let $\mathbf{r}$ be a probability on $\DD$ such that the canonical
process $\xx$ on $\DD$ is a nice semimartingale 
\begin{equation}\label{eq-18}
    \xx=\xx_0+B+M^\mathbf{r},\quad \mathbf{r}\as
\end{equation}
where $B$ is an absolutely continuous process and $M^\mathbf{r}$
is a local $\mathbf{r}$-martingale. Suppose also that the
quadratic variation and the jump compensator are absolutely
continuous. More precisely, there exists a nonnegative adapted
process $a$ such that $\II01 a_t\,dt<\infty,$ $\mathbf{r}\as$ and
$$
    d[\xx,\xx]^c_t=a_t\,dt,\quad \mathbf{r}\as
$$
and the dual predictable projection $\overline{\ell}$ of the jump
measure $\sum_{0\le s\le t}\delta_{(s,\Delta \xx_s)}$ has the
following form
$$
    \overline{\ell}_t(dtd q)=dt\ell_t(d q), \quad \mathbf{r}\as
$$
This means that $\ell_t=\ell(t,\xx_{[0,t)};\cdot)$ is a predictable
nonnegative measure on $\RR_*:=\RR\setminus\{0\}$ such that
$$
    \sum_{0\le s\le t}f(s,\xx_{[0,s)};\Delta \xx_s)
    = \int_{[0,t]\times\RR_*}f(s,\xx_{[0,s)}; q)\,ds\ell_s(d q)
    +M^f_t
$$
    where $M^f$ is a local $\mathbf{r}$-martingale and this
    decomposition is valid
for any measurable function $f$ such that
$\int_{\ii\times\RR_*}|f(t,\xx_{[0,t)}; q)|\,dt\ell_t(d q)<\infty,$
$\mathbf{r}\as$
\\
It is also assumed that
\begin{equation}\label{eq-15}
 \int_{\ii\times\RR_*}\theta(\alpha| q|)\,dt\ell_t(d q)<\infty,\ \forall \alpha\ge0\quad \mathbf{r}\as
\end{equation}
where $\theta(a):=e^a-a-1,$ $a\in\RR$ already appeared at \eqref{eq-45}.

\begin{lemma}\label{res-07}
Let $\mathbf{r}$ be as above and $\mathbf{p}$ be a probability on $\DD$ such that
$H(\mathbf{p}|\mathbf{r})<\infty.$ Then, $\xx$ is also a nice
$\mathbf{p}$-semimartingale.
\end{lemma}

\begin{remarks}\
\begin{enumerate}
    \item Girsanov's theorem tells us that if $\xx$ is an
$\mathbf{r}$-semimartingale and $\mathbf{p}\ll\mathbf{r},$ then
$\xx$ is also a $\mathbf{p}$-semimartingale. This lemma tells us
that the property of being a \emph{nice} semimartingale is also
hereditary under the stronger condition that
$H(\mathbf{p}|\mathbf{r})<\infty.$
    \item In case when no jump occurs and the $\mathbf{r}$-semimartingale is
built on a Brownian filtration, it is well-known that Lemma
\ref{res-07} is still valid with the weaker assumption that
$\mathbf{p}\ll\mathbf{r}$ instead of
$H(\mathbf{p}|\mathbf{r})<\infty.$ This follows from Girsanov's
theorem and a martingale representation theorem.
	\item  The assumption $H(\mathbf{p}|\mathbf{r})<\infty$ is not
    very restrictive. Indeed, $\mathbf{p}\ll\mathbf{r}$ means that
    $d\mathbf{p}/d\mathbf{r}\in L^1(\mathbf{r}),$ while
    $H(\mathbf{p}|\mathbf{r})<\infty$ means that
    $(d\mathbf{p}/d\mathbf{r})\log_+\left(d\mathbf{p}/d\mathbf{r}\right)\in L^1(\mathbf{r}).$
\item
For more details about extensions of this result, see \cite{Leo11a}.
\end{enumerate}
\end{remarks}

\proof The proof is based on the variational representation
\begin{equation}\label{eq-16}
     H(\mathbf{p}|\mathbf{r})=\sup\{E_{\mathbf{p}}u-\log E_{\mathbf{r}}e^u;
    u \textrm{ measurable}: E_{\mathbf{r}}e^u<\infty\}
\end{equation}
of the relative entropy which holds true for any probability
measure $\mathbf{p}$ such that $H(\mathbf{p}|\mathbf{r})$ is
finite, see for instance \cite[Lemma 3.1]{Leo11a} for a proof.
\\
Let $h$ belong to the space $\mathcal{S}$ of all simple
predictable processes: $$h_t=h_0\1_{\{0\}}(t)+\sum_{i=1}^k
h_i\1_{(T_{i},T_{i+1}]}(t)$$ with $k$ a finite integer,
$h_i\in\sigma(\xx_{[0,T_i)}),$ $|h_i|<\infty$ and $0\le T_1\le
\cdots \le T_{k+1}=1$ an increasing sequence of stopping times.
Its stochastic integral with respect to $M^\mathbf{r}$ is $h\cdot
M^\mathbf{r}_t=\sum_{i=1}^kh_i(M^\mathbf{r}_{T_{i+1}\wedge
t}-M^\mathbf{r}_{T_i\wedge t}),$ $t\in\ii$ and the stochastic
exponential of $h\cdot M^\mathbf{r}$ is
$$
    \mathcal{E}(h\cdot M^\mathbf{r})_t
    =\exp \left(h\cdot M^\mathbf{r}_t
    -\II0t \frac{h_s^2}2 a_s\,ds -\int_{[0,t]\times\RR_*}\theta(h_s q)\,ds\ell_s(d q)\right)
$$
Under the assumption \eqref{eq-15}, the integrals in the
exponential are finite $\mathbf{r}\as$ so that the sequence of
stopping times $\inf\{t\in\ii; \II0t \frac{h_s^2}2 a_s\,ds
+\int_{[0,t]\times\RR_*}\theta(h_s q)\,ds\ell_s(d q)\ge k\}$
tends to infinity $\mathbf{r}\as$ as $k$ tends to infinity. It
follows that $ \mathcal{E}(h\cdot M^\mathbf{r})$ is a positive
supermartingale and in particular that: $E_{\mathbf{r}}
\mathcal{E}(h\cdot M^\mathbf{r})_1\le 1.$ Therefore, for any $h$
in $\mathcal{S},$ $\log E_{\mathbf{r}} \mathcal{E}(h\cdot
M^\mathbf{r}_1)\le 0$ and with \eqref{eq-16} we obtain
\begin{eqnarray*}
    E_{\mathbf{p}}(h\cdot M^\mathbf{r}_1)
    &\le& H(\mathbf{p}|\mathbf{r})
    + E_{\mathbf{p}}\left(\II01 \frac{h_t^2}2 a_t\,dt
    +\int_{\ii\times\RR_*}\theta(h_t q)\,dt\ell_t(d q)\right)\\
   &\le& H(\mathbf{p}|\mathbf{r})
    +
    \int_{\ii\times\DD}\Phi(t,\eta;h_t(\eta))\,\overline{\mathbf{p}}(dtd\eta)
\end{eqnarray*}
where
$$
    \overline{\mathbf{p}}(dtd\eta)=dt\mathbf{p}(d\eta)
$$
and for all $t\in\ii,\eta\in\DD, x\in\RR,$
$$
    \Phi(t,\eta;x):=a_t(\eta)x^2/2+\int_{\RR_*}\theta(| q x|)\ell(t,\eta;d q).
$$
A standard convexity argument (note that $\theta(|x|)$ is a convex
nonnegative even function) proves that the gauge functional
$$
    |h|_\mathbf{p}:=\inf\left\{\alpha>0;  \int_{\ii\times\DD}\Phi(t,\eta;h_t(\eta)/\alpha)\,\overline{\mathbf{p}}(dtd\eta)\le
    1\right\}\in[0,\infty)
$$
is a seminorm on $\mathcal{S}.$ Considering $h/|h|_\mathbf{p}$ and
$-h/|h|_\mathbf{p}$ in the above inequality,  it is easy to deduce
that
\begin{equation*}
|E_{\mathbf{p}}(h\cdot M^\mathbf{r}_1)| \le
(H(\mathbf{p}|\mathbf{r})+1)|h|_\mathbf{p},\quad \forall h\in
\mathcal{S}.
\end{equation*}
This means that if $H(\mathbf{p}|\mathbf{r})<\infty,$ $h\mapsto
E_{\mathbf{p}}(h\cdot M^\mathbf{r}_1)$ is a
$|\cdot|_\mathbf{p}$-continuous linear form on $\mathcal{S}.$ But,
$|\cdot|_\mathbf{p}$ is the seminorm of an Orlicz space and by
assumption \eqref{eq-15}, $\int
\Phi(ah)\,d\overline{\mathbf{p}}<\infty$ for all $a\ge0$ and $h\in
\mathcal{S}.$ This implies that $\mathcal{S}$ is a subspace of the
``small'' Orlicz space   
$S^\Phi(\overline{\mathbf{p}}):=\{f:\ii\times\DD\to\RR, \textrm{
measurable}, \int
\Phi(t,\eta;af_t(\eta))\,\overline{\mathbf{p}}(dtd\eta)<\infty,\forall
a\ge0\}$ whose dual representation is well-known, see
\cite{RaoRen}: There exists a measurable function $k$ on
$\ii\times\DD$ which stands in the ``large'' Orlicz space
$\{k:\ii\times\DD\to\RR, \textrm{ measurable}, \int
\Phi^*(t,\eta;a_ok_t(\eta))\,\overline{\mathbf{p}}(dtd\eta)<\infty,
 \textrm{ for some }a_o>0\}=:L^{\Phi^*}(\overline{\mathbf{p}})\subset
L^{1}(\overline{\mathbf{p}})$ associated with the convex
conjugates $\Phi^*(t,\eta;\cdot)$ of $\Phi(t,\eta;\cdot),$ such
that
\begin{equation}\label{eq-17}
    E_{\mathbf{p}}(h\cdot M^\mathbf{r}_1)
    =\int_{\ii\times\DD} k_t(\eta)h_t(\eta)\,\overline{\mathbf{p}}(dtd\eta),
    \quad \forall h\in\mathcal{S}.
\end{equation}
Since $h$ is predictable, we also have
    $
    \int
        kh\,d\overline{\mathbf{p}}=
   \II01 E_\mathbf{p} (k_th_t)\,dt
    =E_\mathbf{p}\II01 E_\mathbf{p}(k_t\mid \xx_{[0,t)})h_t\,dt
    $
and taking $\tilde{b}_t=E_\mathbf{p}(k_t\mid \xx_{[0,t)})$ we see
with \eqref{eq-17} that
$$
     E_{\mathbf{p}}\left(\II01 h_t\,dM^\mathbf{r}_t
     -\II01 h_t\tilde{b}_t\,dt\right)=0,
     \quad\forall  h\in\mathcal{S}.
$$
It follows that $M^\mathbf{p}_t:=M^\mathbf{r}_t-\widetilde{B}_t$
with $\widetilde{B}_t:=\II0t\tilde{b}_s\,ds$ is a local
$\mathbf{p}$-martingale and with \eqref{eq-18} we finally obtain
that
    $$\xx=\xx_0+B+\widetilde{B}+M^\mathbf{p},\quad \mathbf{p}\as$$
where $B+\widetilde{B}$ has absolutely continuous sample paths
$\mathbf{p}\as$
\endproof

Let us go back to $R$ and $P$ given at \eqref{eq-10}.

\begin{definition}[The class $\UR$] Let the reference
Markov process $R$ be given.  We say that the measurable function
$u:\iX\to\RR$ is in the class $\UR$ (with respect to $R$):
$u\in\UR,$ if
\begin{enumerate}[(a)]
    \item  $u\in\dom\LR ;$
    \item $d[u_t(X_t),u_t(X_t)]^c\ll dt,$ $R\as;$
    \item the predictable dual projection $\overline{\ell^u}$ of
$\sum_{t\in\ii}\delta_{(t,\Delta u_t(X_t))}$ satisfies
    $\overline{\ell^u}(dtd q)=dt\ell^u_t(d q)$ and
    $\int_{\ii\times\RR_*}\theta(\alpha| q|)\,dt\ell^u_t(d q)<\infty$
for all $\alpha\ge0,$ $R\as$
\end{enumerate}
\end{definition}

In other words, $u\in\UR$ if  the process $u(t,X_t)$ is a
$R$-semimartingale and its law
$\mathbf{r}\in\mathrm{P}(\DD)$ meets the assumptions of Lemma \ref{res-07}.
 
\begin{remark}\label{rem-01}
For the class $\UR,$ we have in mind $\CCc$ when $R$ is such that the canonical process $X$ is a \emph{nice}  $R$-semimartingale with its values in $\XX=\Rd$. Indeed, at least in the continuous case when no exponential moments of $\ell^u$ are required,  It\^o's formula immediately implies that $\CCc\subset\UR.$ 
\\
Otherwise, if $X$ is not a nice  $R$-semimartingale, then it might happen that $\UR$ reduces to the constant functions.
\end{remark}

A useful result is the following
\begin{proposition}\label{res-13}
Let us assume that $H(P|R)<\infty.$
 Then any $u\in \UR$ is also in $\dom\LP .$
\end{proposition}

\proof Let $\Psi:\OO\to\DD$ be the application
$\Psi=(u_t(X_t))_{t\in\ii}.$ The measure $\mathbf{r}=\Psi\pf
R\in\mathrm{P}(\DD)$ is the law of the process
$(u_t(X_t))_{t\in\ii}$ when the canonical process is governed by
$R\in\PO.$ By the definition of the class $\UR,$  $\mathbf{r}$
satisfies the assumptions of Lemma \ref{res-07}. Let
$\mathbf{p}=\Psi\pf P$ be the law of  $(u_t(X_t))_{t\in\ii}$ under
$P\in\PO.$ By contraction of the relative entropy (an easy
consequence of \eqref{eq-16}), we have
$H(\mathbf{p}|\mathbf{r})=H(\Psi\pf P|\Psi\pf R)\le
H(P|R)<\infty.$ This is the second assumption of Lemma \ref{res-07}, and this lemma tells us that $(u_t(X_t))_{t\in\ii}$ is a nice
$P$-semimartingale, i.e.\ $u\in\dom\LP .$
\endproof

\begin{lemma}\label{res-12}
Let $P\in\PO$ be specified by \eqref{eq-10} with $\inf V>-\infty$
and $\ff,\gg\ge0.$  Then, for $H(P|R)<\infty,$ it is sufficient that
$\IX \ff^2\log_+^{p}(\ff)\,dm<\infty$ and $\IX
\gg^2\log_+^{p}(\gg)\,dm<\infty$ for some $p>1.$
\end{lemma}

\proof Since $m$ and $R$ are bounded positive measures, only the
large values of the functions are important as regards
integrability issues. As $\exp(-\II01 V_t\,dt)$ is bounded, all we
have to show is that if two nonnegative functions $F=\ff(X_0)$ and
$G=\gg(X_1)$ satisfy $\int F^2\log_+^{p}(F)\,dR<\infty$ and $\int
G^2\log_+^{p}(G)\,dR<\infty$ with $p>1,$ then $\int FG\log_+(FG)\,dR<\infty.$
\\
For all $x,y\ge 0,$ we have $xy\le x^2\log_+x$ when $y\le x\log_+
x$ and in the alternate case when $y\ge x\log_+x,$ we see that for
any $0<q<1$ and $y\ge y_q$ large enough, $x\le y (\log_+y)^{-q}.$
Hence,
    $$
    xy\le x^2\log_+x +y^2(\log_+y)^{-q}, \quad\forall x\ge0, y\ge
    y_q.
    $$
Now, for $F,G$ large enough we have
\begin{eqnarray*}
  FG\log_+(FG)
  &\le& (F\log_+F) G+F G\log_+G \\
  &\le& F^2\log_+F +G^2\log_+G
  +\frac{(F\log_+F)^2}{\log_+^{q}(F\log_+F)}
    +\frac{(G\log_+G)^2}{\log_+^{q}(G\log_+G)}\\
  &\le& 2F^2\log_+^{2-q}F +2G^2\log_+^{2-q}G
\end{eqnarray*}
which completes the proof of the lemma.
\endproof

Let $\chi(a)$ be a Young function, then 
\begin{equation}\label{eq-30}
\gamma_f(a):=\chi(|a|\log_+|a|)\quad \textrm{and}\quad\gamma_g(b):=\chi^*(|b|\log_+|b|)
\end{equation}
are also  Young functions. Clearly, 
$ab\log_+(ab)\le (a\log_+a)b+ a(b\log_+b)\le
2[\gamma_f(a)+\gamma_g(b)]$ for any large enough positive numbers $a,b.$ Therefore, if $f\in L^{\gamma_f}$ and $g\in L^{\gamma_g},$ then $fg\in L\log L.$

Gathering our last results leads us to the following statement.

\begin{theorem}\label{res-08}
Let $P$ be the generalized $h$-process given at \eqref{eq-10} with
$\inf V>-\infty$ and the functions $\ff$ and $\gg$ such that one of the following conditions is satisfied:
\begin{enumerate}[(i)]
\item 
$\ff\in L^{\gamma_f}(m)$ and $\gg\in L^{\gamma_g}(m)$ where $\gamma_f$ and $\gamma_g$ satisfy \eqref{eq-30};
\item
$\IX \ff^2\log_+^{p}(\ff)\,dm<\infty$ and $\IX
\gg^2\log_+^{p}(\gg)\,dm<\infty$ for some $p>1.$
\end{enumerate}
 Then, $H(P|R)<\infty$
and any function $u\in\dom\LR $ which is in the class
$\UR$ is also in the extended domain $\dom\LP $
associated with $P.$
\end{theorem}

As  particular cases of condition (i) above, we have
$\ff\in L^\infty(m),\gg\in\LlogL$
and $\ff\in \LlogL,\gg\in L^\infty(m).$ Condition (ii) is a slight improvement of condition (i) with $\chi(x)=x^2.$

\subsection*{The stochastic derivative of $P$}

Let us start saying some words about the carr\'e du champ operator
$\GR$ of a Markov process $P.$ It is a general result of the
theory of stochastic processes that the product of two real
semimartingales is still a semimartingale. More precisely, if $Y$
and $Z$ are semimartingales, then $$YZ=YZ_0+\int Y_-dZ+\int
Z_-dY+[Y,Z]$$ where $[Y,Z]_t$ is the limit along refining finite
partitions of the time interval by means of stopping times: $0\le
T_1\le\cdots\le T_k=1,$ of the cross variation
    $\sum_{i} (Y_{T_{i+1}\wedge t}-Y_{T_i\wedge t})(Z_{T_{i+1}\wedge t}-Z_{T_i\wedge t}).$
It is a remarkable result that $[Y,Z]$ is again a semimartingale.
Its compensator is denoted by $\langle Y,Z\rangle,$ this means
that
$$
    [Y,Z]=\langle Y,Z\rangle +M^{Y,Z}
$$
where $\langle Y,Z\rangle$ is a predictable bounded variation
process and $M^{Y,Z}$ is a local martingale. Nevertheless, the
product of two \emph{nice} semimartingales might not be nice
anymore. Let $Y$ and $Z$ be nice. Clearly, the stochastic
integrals $\int Y_-dZ$ and $\int Z_-dY$ are nice so that $YZ$ is
nice if and only if $\langle Y,Z\rangle$ is absolutely continuous.

\begin{definition}[Carr\'e du champ operator]
Let $u$ and $v$ be two measurable real functions on $\iX.$
Going back to the canonical process $X$ on $\OO,$ suppose that the
processes $u(X)=(u_t(X_t))_{t\in\ii}$ and
$v(X)=(v_t(X_t))_{t\in\ii}$ are $P$-semimartingales such that
$\langle u(X),v(X)\rangle$ is absolutely continuous $P\as$ Then,
we say that the couple of functions $(u,v)$ is in the domain
$\dom\GP$ of the carr\'e du champ operator $\GP$ which is defined
by
\begin{equation*}
  d\langle u(X),v(X)\rangle_t=:  \GR(u,v)(t,\lg Xt)\,dt,
  \quad P\as
\end{equation*}
This identity determines the function
$(t,x)\in\iX\mapsto\GP(u,v)(t,x)\in\RR,$ $dtP_t(dx)$-almost
everywhere.
\end{definition}
As a direct consequence of this definition, we obtain the following result which is often used as a definition of $\GP.$

\begin{proposition}\label{res-11}
Let $u$ and $v$ be two continuous functions on $\iX$ such that $u, v$ and their product $uv$ belong to $\dom \LP .$ Then, $(u,v)\in\dom\GP$ and
\begin{equation*}
\GP(u,v)=\LP (uv)-u \LP v-v \LP u.
\end{equation*}
\end{proposition}
\begin{proof}
We denote $U_t=u(t,X_t)$ and $V_t=v(t,X_t).$ By hypothesis, we have
$dU_t=\LP u(t,X_t)\,dt+dM^u_t,$ $dV_t=\LP v(t,X_t)\,dt+dM^v_t$ and $d(UV)_t=\LP(uv)(t,X_t)\,dt+dM^{uv}_t$ where $M$ stands for any local $R$-martingale. Therefore,
\begin{eqnarray*}
d[U,V]_t
	&=& d(UV)_t-U_tdV_t-V_tdU_t\\
	&=& [\LP(uv)-u\LP(v)-v\LP(u)](t,\lg Xt)\,dt +dM_t
\end{eqnarray*}
with $dM_t=dM^{uv}_t-U_tdM^v_t-V_tdM^u_t.$ Hence, $d\langle U,V\rangle_t=[\LP(uv)-u\LP v-v\LP u](t,\lg Xt)\,dt$, which is the announced result.
\end{proof}

There are no tractable general conditions on $P$ which imply that
$\langle u(X),v(X)\rangle$ is absolutely continuous $P\as$
whenever $u,v\in\dom\LP .$ Counterexamples are known, see
\cite{Moko89}; $u,v\in\dom\LP $ doesn't imply in general
that $(u,v)\in\dom\GP.$ Some additional assumptions are needed.

\begin{theorem}\label{res-09}
Let the $h$-process $P$ and the function $g_t(x)$ be defined by \eqref{eq-10} and  \eqref{eq-11}.
Let the hypotheses of Theorem \ref{res-06} and Proposition \ref{res-13} be satisfied:
\begin{enumerate}[(i)]
\item
$R\in\PO$ is a stationary Markov process with invariant law $m=R_t\in\PX$ for all
$t\in\ii;$
\item
$\gamma$ is a Young function which satisfies \eqref{eq-05} and $\gamma^*$ is its convex conjugate;
\item
$V$ is a measurable function on $\iX$ which is
bounded below and is such that $\{\gamma^*( V_t);t\in\ii\}$ is uniformly
integrable in $L^1(m);$
\item
$\gg$ is a nonnegative function on $\XX$ 
in $\Lg.$
\end{enumerate}
We also assume that $\ff$ and $\gg$ satisfy the hypotheses of Theorem \ref{res-08} to insure that $H(P|R)<\infty.$
\\
 Then, 
$\UR\subset\dom \LP\subset\dom L^P$ and for all
$u\in\UR$ which satisfies for almost all $t\in[0,1)$ and $m$-almost all $x,$
\begin{equation}\label{eq-23}
    \sup_{s\in[t,t+h_o]}
E_R(|\LR u_s|^p\mid X_t=x)<\infty,\quad\textrm{for some }h_o>0 \textrm{ and }p>1,
\end{equation}
we have
 $$(g,u)\in\dom\GR$$ and
\begin{equation*}
  \LP u(t,x)=  L^Pu(t,x)=L^Ru(t,x) +\frac{\GR(g,u)(t,x)}{g_t(x)},
    \quad  dtP_t(dx)\ae
\end{equation*}
where no division by zero occurs since $g_t>0,$ $P_t\as$
\end{theorem}

\proof Let $u$ be in $\UR,$ then we know by Theorems \ref{res-02}
and \ref{res-08} that
\begin{equation}\label{eq-22}
    u\in\dom \LP \subset\dom L^P.
\end{equation}
With \eqref{eq-14f} we see that for all $0\le t< t+h\le1$ and $P_t$-almost all $x,$
\begin{multline*}
    E_P(u_{t+h}(X_{t+h})-u_t(x)\mid X_t=x)
   \\ =E_R\left( g_t(x)^{-1}[u_{t+h}(X_{t+h})-u_t(x)]\exp\left(-\II
    t{t+h}V_r(X_r)\,dr\right)g_{t+h}(X_{t+h})\mid X_t=x\right)
\end{multline*}
with $g_t(x)>0,$ $P_t(dx)\as$ We write for simplicity
$u_s(X_s)=U_s,$ $g_s(X_s)=G_s,$ $V_s=V_s(X_s),$
$D^hU_t=u_{t+h}(X_{t+h})-u_t(x),$ $D^hG_t=g_{t+h}(X_{t+h})-g_t(x)$
and $D^hF_t=\II t{t+h}V_r(X_r)\,dr.$ The inner term in the right-hand side
expectation is
\begin{eqnarray}
    &&g_t(x)^{-1}D^hU_t e^{-D^hF_t}G_{t+h}\nonumber\\
    &=&D^hU_t(1+[e^{-D^hF_t}-1])(1+D^hG_t/ g_t(x))\nonumber\\
    &=&D^hU_t+D^hU_tD^hG_t/ g_t(x)+[e^{-D^hF_t}-1][D^hU_t+D^hU_tD^hG_t/
    g_t(x)].\label{eq-20}
\end{eqnarray}
As it is assumed that $u\in\dom\LR ,$ $(U_r)_{r\in\ii}$
is a $R$-semimartingale. Since its sample paths are in $\DD,$ they
are bounded $R\as$ and the sequence of stopping times
$\inf\{r\in\ii; |U_r|+G_r\ge k\}$ converges $R\as$ to infinity.
Therefore, we can assume without loss of generality that $U$ and
$G$ are bounded without introducing integration times. 
\\
The contribution of the first term $D^hU_t$ of \eqref{eq-20} is
well understood. Since $u\in\dom L^R,$ we have
\begin{equation}\label{eq-19}
    \Limh \frac 1h E_R^x D^hu_t=L^Ru(t,x)
\end{equation}
where we denote $E_R^x=E_R(\cdot\mid X_t=x)$ for simplicity.
\\
Let us control, the last term of \eqref{eq-20}. As $G$ and $U$ can
be assumed to be bounded, $DU_t$ and $DG_tDU_t$ are also bounded.
Hence, $DU_t+DG_tDU_t/g_t(x)$ is bounded and by right continuity
of the sample paths, it tends to zero $R\as.$ By dominated
convergence, we obtain
$$
    \Limh E_R^x \gamma(DU_t+DG_tDU_t/g_t(x))=0.
$$
On the other hand,
    $|e^{-D^hF_t}-1|
    = |[e^{-D^hF_t}-1]/(-D^hF_t)|\, |D^hF_t|
    \le e^{\lo h}|\II t{t+h}V_r\,dr|
    $
and
\begin{multline*}
   E_R^x \gamma^*\left(\frac 1h [e^{-D^hF_t}-1]\right)
    \le c_{\gamma^*,\lo} E_R^x \gamma^*\left(\frac 1h\II t{t+h}V_r\,dr\right)
   \le c_{\gamma^*,\lo} E_R^x \frac 1h\II t{t+h} \gamma^*( V_r)\,dr\\
    = c_{\gamma^*,\lo} \frac 1h\II t{t+h}E_R^x \gamma^*( V_r)\,dr
    \le c_{\gamma^*,\lo}\sup_{r\in\ii}E_R^x \gamma^*( V_r)<\infty.
\end{multline*}
It follows with H\"older's inequality that
\begin{multline}\label{eq-21}
   \Limh \frac 1h E_R^x \left|[e^{-D^hF_t}-1][D^hU_t+D^hU_tD^hG_t/
   g_t(x)]\right|
  \\  \le 2\Limh \big\|\frac 1h [e^{-D^hF_t}-1]\big\|_{L^{\gamma^*}(R^x)}  \big\|DU_t+DG_tDU_t/g_t(x)\big\|_{L^{\gamma}(R^x)} 
    =0.
\end{multline}

Let us look at $D^hU_tD^hG_t$ coming from the second term of
\eqref{eq-20}. By means of basic stochastic calculus we arrive at
\begin{equation*}
    DG_tDU_t
    =\II t{t+h} (G_r-G_t)\,dU_r+\II t{t+h}(U_r-U_t)\, dG_r
        +[G,U]_{t+h}-[G,U]_t.
\end{equation*}
With $dU_r=\LR u_r\,dr+dM^u_r$ and
$dG_r=\LR g_r\,dr+dM^g_r=V_rG_r\,dr+dM^g_r$ where we
relied on Theorem \ref{res-06} in last equality, taking the
expectation leads us to
\begin{multline*}
  E_R^x( DG_tDU_t)
  \\  =\underbrace{E_R^x\II t{t+h} (G_r-G_t)\,\LR u_r\,dr}_{A_h}+\underbrace{E_R^x\II
  t{t+h}(U_r-U_t)\,V_rG_r\,dr}_{B_h}
        +\underbrace{E_R^x([ G,U]_{t+h}-[ G,U]_t)}_{C_h}.
\end{multline*}
Let us control $A_h, B_h$ and $C_h$. By H\"older's  inequality with $1/p+1/q$ and $q\ge1,$
\begin{equation*}
    A_h\le \left(E_R^x\II t{t+h} |G_r-G_t|^q\,dr\right)^{1/q} \left( E_R^x\II
    t{t+h}|\LR u_r|^p\,dr\right)^{1/p}.
\end{equation*}
But $ E_R^x\II t{t+h} |G_r-G_t|^q\,dr=o(h)$ since
$\{G_r;r\in\ii\}$ is bounded and $G$ is right
continuous. We also obtain,
    $E_R^x\II t{t+h}|\LR u_r|^p\,dr=\II t{t+h}E_R^x|\LR u_r|^p\,dr=O(h),$
under the condition that \eqref{eq-23} holds.
It follows that $A_h=o(h)^{1/q}O(h)^{1/p}=o(h).$
\\
Let us control $B_h.$ We can take $U$ bounded and we already know
by Lemma \ref{res-05} that $\{V_tG_t;t\in\ii\}$ is uniformly
integrable. Since $U$ is right continuous, it follows that
$B_h=o(h).$
\\
We know by \eqref{eq-22} that the limit
\begin{multline*}
\Limh \frac1h E_R^x \big\{D^hU_t+D^hU_tD^hG_t/ g_t(x)\\+[e^{-D^hF_t}-1][D^hU_t+D^hU_tD^hG_t/
    g_t(x)]\big\}=: L^Pu(t,x)
\end{multline*}
exists. We have also shown \eqref{eq-19} and \eqref{eq-21} which
imply that, $dtP_t(dx)\ae:$
\begin{eqnarray*}
     L^Pu(t,x)&=&L^Ru(t,x)+g_t(x)^{-1} \Limh \frac1h C_h\\
    &=& L^Ru(t,x)+ g_t(x)^{-1} \Limh \frac1h E_R^x([ G,U]_{t+h}-[ G,U]_t)\\
   &=& L^Ru(t,x)+ g_t(x)^{-1} \Limh \frac1h E_R^x(\langle G,U\rangle_{t+h}-\langle G,U\rangle_t)
\end{eqnarray*}
and in particular that the limit $\Limh \frac1h E_R^x(\langle
G,U\rangle_{t+h}-\langle G,U\rangle_t)$ exists. Since this is true
for all $t$ and $x,$ this shows that $(g,u)$ belongs to the domain
of $\GR.$ We conclude noticing that by definition $\Limh \frac1h
E_R^x(\langle G,U\rangle_{t+h}-\langle
G,U\rangle_t)=\GR(g,u)(t,x).$
\endproof

We note for future use the following result.
\begin{corollary}\label{res-19}
Under the assumptions of Theorem \ref{res-09}, we have
\begin{equation}\label{eq-39}
\GR(g,u)(t,x)=\Lim h\frac1h E_R\Big([g_{t+h}(X_{t+h})-g_t(x)][u_{t+h}(X_{t+h})-u_t(x)]|X_t=x\Big),\ dtm(dx)\ae
\end{equation}
The product $ug$ is in $\dom\LR$ and  $\GR(g,u)=\LR(gu)-g\LR u-u\LR g.$
\end{corollary}

\begin{proof}
The identity \eqref{eq-39} has been proved during the previous proof of Theorem \ref{res-09}. Next assertion follows from $D(GU)=DGDU+UDG+GDU$ and the convergences which are implied by $u,g\in\dom\LR$ and $(g,u)\in\dom\GR.$
\end{proof}
 
\begin{remark} 
Let us also remark that applying Lemma \ref{res-07} to the quadratic variation $[u(X)],$ under the assumption  $H(P|R)<\infty$ we see that $d\langle u(X)\rangle^R_t\ll dt,$ $R\as$ implies that $d\langle u(X)\rangle^P_t\ll dt,$ $P\as$ It follows that $$\dom\GP\subset\dom\GR$$ in the sense that we consider $dtP_t(dx)\ae$-defined functions instead of $dtm(dx)\ae$-defined functions. 
\\
In the special case when $X$ is continuous $R\as$, we also have $\langle u(X)\rangle^P=\langle u(X)\rangle^R,$ $P\as,$ which implies that $\GP(u,v)(t,x)=\GR(u,v)(t,x),\ dtP_t(dx)\ae,$ $(u,v)\in\dom\GP.$
\end{remark}

\section{Continuous diffusion processes on $\Rd$}\label{sec-diffusion}

In this section we examplify the previous abstract results with simple continuous diffusion processes on $\Rd.$

\subsection*{The reference process $R$}

The reference process $R$ is the law of a Markov continuous diffusion process on the state space $\XX=\Rd$ which admits an invariant probability measure $m.$ To fix the ideas, we assume in the whole section that it is the solution of the stochastic differential equation (SDE)
\begin{equation*}
X_t=X_0+\II0t b(X_s)\,ds+\II0t \sigma(X_s)\,dW_s,
\quad t\in\ii
\end{equation*}
where $W$ is a $\Rd$-valued Wiener process, $b:\Rd\to\Rd$ and $\sigma:\Rd\to M_{d\times d}$ are locally Lipschitz functions which are respectively vector-valued and matrix-valued. We also assume that $R\as,$ $X$ doesn't explode on the time interval $\ii.$

\begin{result}\label{res-16}
Under these hypotheses on $R,$ it is known that 
  $R$ is the unique solution of the martingale problem $\MP(\mathcal{L},\mathcal{C};\mu_o)$ in the sense of Definition \ref{def-02} with the initial measure $\mu_o= m$ and the generator $\LR$ given for all $u\in \mathcal{C}=\CCc$ by
$$
	\LR u(t,x)=\partial_t u(t,x)+\sum_{i=1}^db_i(x)\partial_{x_i} u(t,x)+
	\frac 12\sum_{i,j=1}^d a_{ij}(x)\partial_{x_i} \partial_{x_j} u(t,x)
$$
where $(a_{ij})_{1\le i,j\le d}=a:=\sigma \sigma^*\in M_{d\times d}.$\\ We denote this martingale problem $\MP(b,a;m).$
\end{result}

\subsection*{Extended gradients}

We introduce the notion of extended gradient. Let $P$ be a solution to the martingale problem $\MP(b^P,a;P_0)$, for some drift vector field $b^P:\iR\to\Rd.$ A simple computation based on Proposition \ref{res-11} gives us
\begin{equation}\label{eq-24}
\GP(\varphi,v)(t,x)=\nabla \varphi_t(x)\cdot a(x)\nabla v(x),\ dtP_t(dx)\ae,\  \varphi\in\CCc, v\in \Cc.
\end{equation}
One proves the Cauchy-Schwarz type inequality
$$
	\left(\II std\langle A,B\rangle_r\right)^2
	\le \II std\langle A,A\rangle_r\II std\langle B,B\rangle_r, \quad 0\le s\le t\le1
$$
with the usual discriminent argument. Let us take $u,v$ in $\Cc$ and $\psi$ a measurable function on $\iX$ such that  $(\psi,u)$ and $(\psi,v)$ are in $\dom\GP$ and such that $\GP(v-u,v-u)=0.$ Then, the above Cauchy-Schwarz inequality implies that  $\GP(\psi_t,u)(t,x)=\GP(\psi_t,v)(t,x),$ $dtP_t(dx)\ae$ Consequently,  the linear operator $u\mapsto\GP(\psi,u)$ only depends on the equivalence class defined by $u\sim v \overset{\mathrm{def}}{\Leftrightarrow} \GP(v-u,v-u)=0, dtP_t(dx)\ae\Leftrightarrow a\cdot\nabla(v-u)=0, dtP_t(dx)\ae,$ and it follows that there exists some vector field $\beta$ on $\ii\times \Rd$ such that $\GP(\psi,\cdot)$ is represented by
\begin{equation}\label{eq-25}
\GP(\psi,v)(t,x)=\beta_t(x)\cdot a(x)\nabla v(x),\ dtP_t(dx)\ae, \quad v\in \mathcal{C}^2(\Rd).
\end{equation}
Moreover, up to $dtP_t(dx)$-a.e.\ equality, there is a unique such $\beta$ with its values in the range of $a.$

Comparing \eqref{eq-24} and \eqref{eq-25}, it is natural to introduce the following definition.
\begin{definition}[Extended gradient]
Let $\psi$ be a measurable function on $\iR$ such that for all $u\in \Cc,$ $(\psi,u)$ is in $\dom\GP.$ The unique vector field $\beta$ which satisfies \eqref{eq-25} and $\beta_t(x)\in \mathrm{Range}\,a(x)$ up to $dtP_t(dx)$-a.e.\ equality is denoted by
$
\beta=\nablat^P \psi
$
and it is called the $P$-extended gradient of $\psi$. 
\\
When no confusion can occur, we simply drop $P$ and write $\nablat^P \psi=\nablat\psi.$
\end{definition}

It is clear with our previous discussion that for any $u\in \Cc,$ $\nablat u$ is the orthogonal projection of $\nabla u$ on the range of the diffusion matrix $a.$ In particular, $\nablat u=\nabla u,$ $dtP_t(dx)\ae,$ when $a(x)$ is invertible for all $x\in\Rd$.

\subsection*{The martingale problem which is solved by $P$}

Now we consider the generalized $h$-process $P$.
We are going to see that $P$ solves a martingale problem $\MP(b+a \beta,a)$ and that the additional drift $\beta$  has the special form 
$$
\beta= \nablat^P \psi,\quad P\as
$$
with $\psi=\log g,$ i.e.
\begin{equation}\label{eq-36}
\psi(t,x):=\log E_R \left[\exp\Big(-\II t1 V_s(X_s)\,ds\Big)\gg(X_1)\mid X_t=x\right],\quad dtP_t(dx)\ae
\end{equation} 
which is well-defined $dtP_t(dx)\ae$ since $g(t,x)>0,$ $dtP_t(dx)\ae,$ but might not be defined $dtm(dx)\ae$ in general.

\begin{lemma}\label{res-14}
Assume that $R$ satisfies the hypotheses of Result \ref{res-16} and  $P$ defined by \eqref{eq-10} satisfies the hypotheses of Theorem \ref{res-09}. Then,
 for all $u\in\UR$ which verifies \eqref{eq-23},  $(\psi,u)$ is in $\dom\GP$ and
$$
\frac{\GR(g,u)}{g}(t,x)=\GP(\psi,u)(t,x),\quad dtP_t(dx)\ae
$$
\end{lemma}

\begin{proof}
Let us denote $Z_t=dP_{[0,t]}/dR_{[0,t]}$. As $Z$ admits a continuous version and $Z_t=\ff(X_0)\exp \left(-\II0tV_s(X_s)\,ds\right)G_t$ with $G_t:=g_t(X_t),$ $G$ also admits a continuous version.
Applying It\^o's formula to the continuous process $\psi_t(X_t)=\log G_t$, we obtain
\begin{equation}\label{eq-31}
d \psi_t(X_t)=\frac{dG_t}{ G_t}-\frac{1}{2}\frac{d\langle G\rangle_t}{ G_t^2}
\quad P\as
\end{equation}
We deduce from this with Theorem \ref{res-09} that for any $u\in\UR$ which verifies \eqref{eq-23},\\
$d\langle \psi(X),u(X)\rangle_t=d\langle G,u(X)\rangle_t/ G_t=[\GR(g,u)/g](t,X_t)\,dt,$ $P\as$ This completes the proof of the lemma.
\end{proof}

\begin{theorem}\label{res-18}
Assume that $R$ satisfies the hypotheses of Result \ref{res-16} and let $P$ be the generalized $h$-process which is  defined by \eqref{eq-10}. Assume also that $\ff,\gg$ and $V$ satisty the hypotheses of Theorem \ref{res-09}.
 \\
Then $P$ is the unique solution in $\left\{Q\in\PO;H(Q|R)<\infty\right\}$ of $$P\in\MP(b+a\nablat^P \psi,a;P_0)$$ with $P_0=\ff g_0\,m$ and where the function $$\psi(t,x):=\log g(t,x) =\log E_R \left[\exp \left(-\II t1 V_s(X_s)\,ds\right)\gg(X_1)\mid X_t=x\right],\quad dtP_t(dx)\ae$$ is defined by \eqref{eq-11} and \eqref{eq-36}. 
\end{theorem}

\begin{proof} Choosing $\UR=\CCc$ in Theorem \ref{res-09}, the assumption \eqref{eq-23}  holds true for all $u\in\CCc.$ The result now follows from Theorem \ref{res-09} and Lemma \ref{res-14}. The assertion $P_0=\ff g_0\,m$ is  \eqref{eq-13}. 
\\
The uniqueness is implied by a  general result of Girsanov's theory since $R$ is the unique solution to its own martingale problem and $H(P|R)<\infty.$ For an entropic point of view  under the present requirement that $H(P|R)<\infty$, see \cite{Leo11a}.
Otherwise, when $P\ll R$ is only assumed  this is a standard result of Girsanov's theory, see \cite{JaShi87}.   
\end{proof}

\subsection*{Kolmogorov diffusion process}

We illustrate this theorem by means of a diffusion process which plays an important role in the area of functional equalities connected with the concentration of measure phenomenon \cite{Ba92,Roy99,Led01,Vill09}.
\\
The Kolmogorov diffusion process is the unique solution of the SDE
\begin{equation*}
dX_t=-\nabla U(X_t)\, dt+ dW_t
\end{equation*}
where  $U$ is a $\mathcal{C}^2$-differentiable function on $\Rd$ such that $Z_U:=\IR e^{-2U(x)}\,dx<\infty.$
This SDE admits the Boltzmann-Gibbs probability measure
\begin{equation*}
m^U(dx):=Z_U^{-1}e^{-2U(x)}\,dx
\end{equation*}
as a reversing measure. We take this reversible Kolmorov diffusion as the reference process $R$. Hence, the initial law is $R_0=m^U$ and
$$
R\in \MP(-\nabla U, \mathrm{Id}).
$$
The generalized $h$-process to be considered here is $P$ specified by \eqref{eq-10} with the  assumptions of Theorem \ref{res-18}. This theorem tells us that $$P\in\MP(-\nabla U+\nablat^P \psi, \mathrm{Id}).$$
In the special case when the potential $V$ is zero, we have for all $0\le t<1,$
$$
g_t(x)=E_R(\gg(X_1)\mid X_t=x)
=[2\pi (1-t)]^{-d/2}\IR \gg(y)\exp\left(\frac{|y-x|^2}{2(1-t)}\right)\,dy.
$$
Therefore, $g\in\mathcal{C}^\infty([0,1)\times\Rd)$ and $g_t$ is positive for all $0\le t<1.$ It follows with 
$$
	\psi_t(x)=\log E_R(\gg(X_1)\mid X_t=x), \quad t\in [0,1), x\in\Rd,
$$
that $\nablat^P \psi_t=\nabla \psi_t$ and 
$$
P\in\MP\big(-\nabla[U-\psi],\mathrm{Id})
$$
and with Theorem \ref{res-06} we see that $\psi$ is a classical solution of the Hamilton-Jacobi-Bellman (HJB) equation
$$
\left\{\begin{array}{ll}
\LR\psi(t,x)
+\frac12 |\nabla \psi_t(x)|^2=0,&\quad t\in [0,1),x\in\Rd\\
\lim_{t\uparrow1}\psi_t(x):=\psi_1(x)=\log\gg(x),&\quad t=1,x\in \left\{\gg>0\right\}
\end{array}•\right.
$$
where $$\LR u(t,x)=\Big(\partial_t -\nabla U(x)\cdot \nabla +\frac12 \Delta\Big) u(t,x).$$

Let us go back to the general case when $V$ is not constant. The positivity improving property of the heat kernel implies that $\psi_t$ is well-defined for all $t\in [0,1).$ But it might not be smooth enough to be a classical solution of the HJB equation:
\begin{equation*}
\left\{\begin{array}{ll}
\LR\psi(t,x)
+\frac12 |\nabla \psi_t(x)|^2-V(t,x)=0,&\quad t\in [0,1),x\in\Rd\\
\lim_{t\uparrow1}\psi_t(x):=\psi_1(x)=\log\gg(x),&\quad t=1,x\in \left\{\gg>0\right\}
\end{array}\right.
\end{equation*}
Because of its semigroup representation \eqref{eq-36}, $\psi$ is a continuous viscosity solution of this equation, see \cite[Thm II.5.1]{FS93} for instance.

\section{Continuous-time Markov chains}\label{sec-chain}

In this section we examplify our results with simple Markov jump processes on a countable discrete space $\XX$ which are analogous to the Kolmogorov diffusion processes. The set of paths is $\OO=D(\ii,\XX).$

\subsection*{The reference process $R$}
Since $\XX$ is a countable discrete space, every function is measurable and continuous. Let $\BX$ denote the space of all real bounded functions on $\XX.$ The first ingredient is a Markov generator
\begin{equation}\label{eq-35}
    \int_\XX [u(y)-u(x)]\,J^0(x;dy),\quad u\in \BX
\end{equation}
where $J^0$ is a kernel of positive measures on $\XX$ such that $J^0(x;\left\{x\right\})=0$  for all $x\in\XX$ and
\begin{enumerate}[(i)]
\item
$J^0(x;\XX)<\infty,$ for all $x\in\XX;$
\item
$J^0$ induces an irreducible process in the sense that $J^0(x;\XX)>0$ for all $x\in\XX$ and for any couple of distinct states $(x,y),$ there exists a finite chain $x=z_1,z_2,\dots,z_n=y$ such that $J(z_i;\left\{ z_{i+1}\right\})>0$ for all $i;$
\item
$J^0$ satisfies the detailed balance condition
\begin{equation}\label{eq-X02}
    m^0(dx)J^0(x;dy)=m^0(dy)J^0(y;dx)
\end{equation}
for some  nonnegative measure $m^0$ on $\XX$ (possibly with an
infinite mass).  
\end{enumerate}
We say that $Q\in\PO$ solves the \emph{martingale problem $\MP(K)$} associated with the predictable jump kernel $K=K(t,X_{[0,t)};dy),$ if 
$$
u(t,X_t)-u(0,X_0)-\II0t ds \IX [u(s,y)-u(s,\lg Xs)]\,K(s,X_{[0,s)};dy),\quad t\in\ii
$$
is a local $Q$-martingale for a large class of functions $u$.

Under the assumption (i), there is a unique law $R^0\in\PO$ which solves the martingale problem with a prescribed initial law and the Markov generator  \eqref{eq-35}: $R^0\in\MP(J^0).$
Under the assumption (iii), the measure $m^0$ is its invariant measure  which is unique (up to
scalar multiplication) under the irreducibility assumption (ii).
\\
The second ingredient is a potential 
$U$ on $\XX$ such that
$	
    Z_U:=\IX e^{-2U}\,dm <+\infty.
$	
 The reference process $R$ is  the law of the Markov jump process with generator
\begin{eqnarray*}
  \LR u(x)&:=&\int_\XX [u(y)-u(x)]\, J(x;dy),\quad u\in\BX \textrm{\quad where} \\
  J(x;dy) &:=& \exp(-[U(y)-U(x)])\, J^0(x;dy)
\end{eqnarray*}
which is well defined for all $u\in\BX$ provided that
\begin{equation*}
    \int_\XX e^{-U(y)}\,J^0(x;dy)<+\infty,\quad \forall x\in\XX,
\end{equation*}
as this last integrability assumption implies that
\begin{equation}\label{eq-37}
J(x;\XX)<\infty,\quad \forall x\in\XX.
\end{equation}
It is easily seen that the Boltzmann-Gibbs probability measure
\begin{equation*}
    m^U(dx):=Z_U^{-1}e^{-2U(x)}\,m^0(dx)
\end{equation*}
is the reversing measure of the jump process $R$  since the detailed balance conditions are
satisfied. Indeed,
\begin{multline*}
  m^U(dx)J(x;dy) 
= e^{-2U(x)}e^{-[U(y)-U(x)]}\,m^0(dx)J^0(x;dy) \\
  = e^{-[U(x)+U(y)]}\,m^0(dx)J^0(x;dy)
  = e^{-[U(x)+U(y)]}\,m^0(dy)J^0(y;dx) 
  = m^U(dy)J(y;dx)
\end{multline*}
where (\ref{eq-X02}) has been used at the last but one equality. Therefore, 
$$
R\in\MP(J;m^U).
$$
Moreover, it is the unique solution of this martingale problem. Indeed, thanks to \eqref{eq-37} it is possible to build a unique strong solution on some auxiliary probability space: a combination of a discrete-time Markov chain with transition probabilities $J(x;dy)/J(x;\XX)$  and independent exponential clocks with frequencies $J(x;\XX),$ $x\in\XX$.

This reference law is sometimes called a \emph{Metropolis dynamics} on the set $\XX.$ It is useful for estimating $m^U$ when the very high cardinality of $\XX$ prevents us from computing the normalizing constant $Z_U.$

\subsection*{The martingale problem which is solved by $P$}

Now we consider the $h$-process $P.$  Applying Theorem \ref{res-09}, we need to compute $\GR(g,u)/g$ for  a large class of functions $u\in\UR.$ We choose this class to be $\BX$ for the following reasons. On  one hand, we can see that $\BX\subset\UR$ because with \eqref{eq-37} it is clear that $\BX\subset\dom\LR$ and for all $u\in\BX$ and $\alpha\ge0,$ $\IiX \exp(\alpha [u(y)-u(\lg Xt)])\,dt J(\lg Xt;dy)<\infty.$ On the other hand, we also see immediately with \eqref{eq-37} that \eqref{eq-23} holds for any bounded function $u$.

\begin{theorem}\label{res-20}
Let $R\in\MP(J;m^U)$ be as above, $P$ be the $h$-process specified at \eqref{eq-10} and assume also that $\ff,\gg$ and $V$ satisty the hypotheses of Theorem \ref{res-09}.
 
Then $P$ is the unique solution in $\left\{Q\in\PO;H(Q|R)<\infty\right\}$ of $\MP(J^P;P_0)$ with $P_0=\ff g_0\,m$ and
$$
 J^P(t,x;dy)=\exp\Big(\psi_t(y)-\psi_t(x)\Big)\,J(x;dy)=\frac{g_t(y)}{g_t(x)}\,J(x;dy),\quad dtP_t(dx)\ae
$$ 
where the function $$\psi(t,x):=\log g(t,x) =\log E_R \left[\exp \left(-\II t1 V_s(X_s)\,ds\right)\gg(X_1)\mid X_t=x\right],\quad dtP_t(dx)\ae$$ is still defined by \eqref{eq-11} and \eqref{eq-36}.
\end{theorem}

\begin{proof}
Corollary \ref{res-19} tells us that for all $u\in\BX,$ $\GR(g,u)=\LR(ug)-u\LR g-g\LR u.$ Hence, $\GR(g,u)(t,x)=\IX [u(y)-u(x)][g_t(y)-g_t(x)]\,J(x;dy),$ $dtm(dx)$ and 
\begin{eqnarray*}
\frac{\GR(g,u)}g(t,x)
&=&\IX [u(y)-u(x)]\left(\frac{g_t(y)}{g_t(x)}-1\right)\,J(x;dy)\\
&=& \IX [u(y)-u(x)]\left(e^{\psi_t(y)-\psi_t(x)}-1\right)\,J(x;dy),
\quad dtP_t(x)\ae
\end{eqnarray*}
We conclude with Theorem \ref{res-09} that $P$ solves the announced martingale problem. The uniqueness statement follows from the general Girsanov theory: because $P\ll R,$ it is inherited from the fact that $R$ is the unique solution of its martingale problem.
\end{proof}

As with the continuous diffusion processes, we see that some gradient of $\psi$ is involved in the shift from the dynamics of $R$ to the dynamics of the $h$-process $P.$ Indeed, denoting $$Du(x;y):=u(y)-u(x)$$ the discrete gradient of $u$ at $x,$ we  have
$$
J^P(x;dy)=\exp \Big( D \psi_t(x;y)\Big)\,J(x;dy).
$$
With Theorem \ref{res-06} we know that $\LR g=Vg.$ If $g$ is time-differentiable and positive on $[0,1)\times\XX$, we deduce that $\psi$ is a classical solution of the following integro-differential  HJB equation
$$
\LR \psi(t,x)+\IX \theta(D \psi_t(x;y))\, J(x;dy)-V(t,x)=0
$$
where $\theta(a):=e^a -a -1$ and $\LR$ is the generator whose value on any $t$-differentiable bounded function $u$ is 
$$
\LR u(t,x)=\partial_t u(t,x)+\IX Du_t(x;y)\,J(x;dy).
$$
In the general case when $g$ might not be time-differentiable and positive on $[0,1)\times\XX$, the semigroup representation of $\psi$ implies that $\psi$ is the unique continuous viscosity solution of the HJB equation
\begin{equation*}
\left\{\begin{array}{ll}
\LR\psi(t,x)
+\IX \theta(D \psi_t(x;y))\, J(x;dy)-V(t,x)=0,&\quad t\in [0,1),x\in\Rd\\
\lim_{t\uparrow1}\psi_t(x):=\psi_1(x)=\log\gg(x),&\quad t=1,x\in \left\{\gg>0\right\}.
\end{array}\right.
\end{equation*}


\end{document}